\documentclass[12pt]{amsart}
\usepackage{amssymb}
\usepackage[utf8]{inputenc}
\usepackage[T1]{fontenc}
\usepackage{hyperref}
\usepackage[mathscr]{eucal}

\newtheorem{lemma}{Lemma}[section]
\newtheorem{proposition}[lemma]{Proposition}
\newtheorem*{proposition*}{Proposition}
\newtheorem{theorem}[lemma]{Theorem}
\newtheorem*{theorem*}{Theorem}
\newtheorem*{conjecture*}{Conjecture}

\newtheorem{corollary}[lemma]{Corollary}
\newtheorem{question}{Question}

\theoremstyle{remark}
\newtheorem*{remark}{Remark}

\theoremstyle{definition}
\newtheorem{definition}[lemma]{Definition}
\newtheorem*{convention}{Convention}
\newtheorem*{notation}{Notation}

\newcommand{\norm}[1]{\lVert #1 \rVert}
\newcommand{\T}{{\mathbb T}}
\newcommand{\R}{{\mathbb R}}
\newcommand{\Z}{{\mathbb Z}}
\newcommand{\N}{{\mathbb N}}

\newcommand{\E}{{\mathbb E}}

\newcommand{\CC}{{\mathcal C}}

\newcommand{\CL}{{\mathcal L}}

\newcommand{\CS}{{\mathcal S}}
\newcommand{\CB}{{\mathcal B}}
\newcommand{\CG}{{\mathcal G}}
\newcommand{\CH}{{\mathcal H}}
\newcommand{\CK}{{\mathcal K}}

\newcommand{\CU}{{\mathcal U}}

\newcommand{\CW}{{\mathcal W}}
\newcommand{\CZ}{{\mathcal Z}}

\newcommand{\one}{{\boldsymbol 1}}

\DeclareMathOperator{\id}{id}

\DeclareMathOperator{\range}{range}

\newcommand{\ua}{\underline{a}}
\newcommand{\ub}{\underline{b}}
\newcommand{\wh}{\widehat}

\newcommand{\MS}{\mathscr{S}}
\newcommand{\MC}{\mathscr{C}}
\newcommand{\MN}{\mathscr{N}}

\DeclareMathOperator{\ad}{Ad}

\begin{document}
\title[Complexity of nilsystems]
{Complexity of Nilsystems and systems lacking nilfactors}
\author{Bernard Host}
\address{BH: Laboratoire d'analyse et de math\'ematiques appliqu\'{e}es, 
Universit\'e  Paris-Est Marne la Vall\'ee \& CNRS UMR 8050\\
5 Bd. Descartes, Champs sur Marne\\
77454 Marne la Vall\'ee Cedex 2, France}
\email{bernard.host@univ-mlv.fr}

\author{Bryna Kra}
\address{BK: Department of Mathematics, Northwestern University \\ 2033 Sheridan Road Evanston \\ IL 60208-2730, USA} 
\email{kra@math.northwestern.edu}

\author{Alejandro Maass}
\address{AM: Departamento de Ingenier\'{\i}a
Matem\'atica, Universidad de Chile
\& Centro de Modelamiento Ma\-te\-m\'a\-ti\-co 
UMI 2071 UCHILE-CNRS \\ Casilla 170/3 correo 3 \\
Santiago, Chili.}
\email{amaass@dim.uchile.cl}

\thanks{The second author was partially supported by NSF grant 
$1200971$ and the third author by Fondap $15090007$ and CMM-Basal grants.  
The authors thank PICS-CNRS for support of the first Franco-Chilean Conference on 
Dynamics and Combinatorics, where this work was begun.}

\keywords{Nilsystems, complexity, topological dynamics}

\subjclass[2010]{54H20, 37A05, 37B40}

\begin{abstract}
Nilsystems are a natural generalization of rotations and 
arise in various contexts, including 
in the study of multiple ergodic averages in ergodic theory, in the 
structural analysis of topological dynamical systems, and in 
asymptotics for patterns in certain subsets of the integers.  
We show, however, that many natural classes in 
both measure preserving systems and topological dynamical systems
contain no higher order nilsystems as factors, meaning that 
the only nilsystems they contain as factors are rotations.  
In the ergodic setting, we show that there are spectral obstructions 
that give rise to this behavior.  In the topological setting, 
nilsystems have a particular type of complexity of polynomial 
growth, where the polynomial (with explicit degree) is an asymptotic both from below and above.
We also deduce several ergodic and topological applications of these results.
\end{abstract}
\maketitle


\section{The role of nilsystems}

\subsection{Nilsystems in ergodic theory}
\label{subsec:nilergo}
In studying multiple ergodic averages in a measure preserving 
system $(X, \CB, \mu, T)$, certain factors with algebraic structure  
occur naturally.
(By a measure preserving system, we mean a Lebesgue probability space 
$(X, \CB, \mu, T)$ endowed with a measurable, measure preserving transformation 
$T\colon X\to X$.)
The simplest case is the invariant $\sigma$-algebra that arises 
in the von Neumann mean ergodic theorem.  
Assuming henceforth that the system is ergodic, the relevant factor for the study of  the average of   $f(T^nx)f(T^{2n}x)$ is the \emph{Kronecker factor} $Z_1(X) = Z_1(X, \CB, \mu, T)$. This is the factor spanned by the eigenvalues of the system, and it is isomorphic to a translation on a compact abelian group endowed with its Haar measure.

For more intricate averages, for example the average of 
$f(T^nx)\cdot f(T^{2n}x)\cdot\ldots\cdot f(T^{sn}x)$,  we need a more 
sophisticated structural analysis of the system.  In~\cite{HK1}, this 
is done by introducing a series of factors $Z_s(X) = Z_s(X, \CB, \mu, T)$ for $s\geq 1$.
The convergence of this average then follows by analyzing the 
averages separately in the factor $Z_s(X)$ and in its orthogonal complement.  

Such structural analysis only becomes useful when the factors have some 
sort of geometric or algebraic structure, and this is the content of the 
structure theorem in~\cite{HK1}.  More precisely, the factor $Z_s(X)$ is the inverse 
limit of all $s$-step  nilsystems that are factors of $(X,\CB,\mu,T)$.
We call a factor of a system that is a nilsystem a \emph{nilfactor} (for the definition of a nilsystem, see Section~\ref{subsec:nilsystems}).

The factors $Z_s(X)$ have since been used to understand other multiple 
ergodic averages (see, for example,  \cite{HK2,L3,FrW}) and 
to prove new  results  on multiple recurrence  (for example  in~\cite{BLL,FrW,Fr2}). 
Nilsystems arise naturally in other contexts such as additive combinatorics; 
as an example, we cite the recent use of nilsequences (introduced in~\cite{BHK}, 
these are functions evaluated on an 
orbit in a nilsystem) in the work on patterns in the primes,  as described in the program laid out in~\cite{GT}.

Our motivation in starting this work was to give examples of  ``simple'' and ``natural systems'' with explicit, nontrivial factors $Z_s$(X) for some $s>1$. 
Since all factors $Z_s(X)$, $s\geq 1$, are trivial for weakly mixing systems, we
restrict our attention to non-weakly mixing ones. 
Of course, the notion of a ``natural system'' is not precisely defined, 
but it is clear that a system obtained by building an  arbitrary  extension of  a given 
nilsystem  is ``artificial.''    Somewhat surprisingly, we found the task of finding non-artificial nilsystems 
harder than expected. 
It turns out that for many natural classes of systems, the factors $Z_s(X)$ coincide with the Kronecker factor $Z_1(X)$; there are spectral obstructions that force this, and
this is explained in Corollary~\ref{cor:spectrum}.

In Sections~\ref{subsec:appli1} and~\ref{subsec:appli2}, we give two applications for 
measure preserving
systems that have no nilfactors other than rotation factors. The first result is on optimal lower bounds for intersections of translates of a set, and the second application is a strengthening of results in~\cite{HK3} and~\cite{C} on the convergence of weighted polynomial multiple averages.

\subsection{Nilsystems in topological dynamics}
\label{subsec:itrotop}
As is often the case, the ergodic questions and results have counterparts in topological dynamics. 
As in the ergodic setting, we refer to topological factors that are nilsystems as 
\emph{topological nilfactors}. 
Topological nilfactors naturally arise in the structural analysis of topological dynamical
systems~\cite{HKM}. 

We recall some definitions (\cite{HKM}, see also~\cite{SY}). Let $(X,T)$ be a transitive topological system, meaning that $X$ is a compact metric space endowed with a homeomorphism $T\colon X\to X$ such that 
the orbit  $\{ T^nx\colon n\in\Z\}$ of some point $x\in X$ is dense.
We can associate to this system an increasing sequence $Z_{\text{\rm top},s}(X)$, $s\geq 1$, of topological factors. 
The first factor  $Z_{\text{top},1}(X)$ is the \emph{maximal equicontinuous factor}, also called  the \emph{topological Kronecker factor}, of $(X,T)$. It is spanned by the continuous eigenfunctions of the system, and  is topologically isomorphic to a translation on a compact abelian group.
 For $s>1$, $Z_{\text{\rm top},s}(X)$ is the inverse limit of all $s$-step topological nilfactors  of $(X,T)$. 
In~\cite{HKM} and~\cite{SY}, it is shown that topological nilfactors can be characterized by 
dynamically defined ``cubic subsets'' of $X^{2^{s+1}}$, leading to a topological analog of the ergodic theoretic structure theorem.

Finding natural topological systems with nilsystems as topological nilfactors also turns out to be nontrivial.  One already has obstructions that arise from the ergodic setting, 
by considering any invariant measure on the system. Secondly, since nilsystems have 
zero entropy, it is only of interest to consider distal systems with zero entropy. 
The relevant property seems to be some sort of complexity, and 
the notion we use is the \emph{topological complexity} inspired in the notion of 
$\varepsilon$--$n$ spanning sets (see for example~\cite{Wa} for the origin of this definition in the work of
Dinaburg and Bowen):
\begin{definition}
Let $(X,T)$ be a topological dynamical system and let $d$ be 
a distance on $X$ defining its topology. For $\varepsilon>0$ and $n\geq 1$, an \emph{$\varepsilon$--$n$ spanning set} for $(X,T)$ is a finite subset $\{x_1,\dots,x_m\}$ of $X$  such that for every $x\in X$ there exists $j\in\{1,\dots,m\}$ such that 
$$
d(T^kx,T^kx_j)<\varepsilon\text{ for every }k\in\{0,\dots,n-1\}.
$$
Let $\MS_{X,T,d}(\varepsilon,n)$ denote the minimal cardinality of an $\varepsilon$--$n$ spanning set of $(X,T)$.  When there is no possible ambiguity, we omit the 
system and metric from the notation and write $\MS(\varepsilon,n)$.  
We call the function $\MS$ the \emph{topological complexity} of the system $(X,T)$. 
\end{definition}

Clearly, this notion depends on the choice of the distance $d$ on $X$.  However, if $d$ and $d'$ are distances on $X$ defining its topology, then for every $\varepsilon>0$ there exist $\eta_1>0$ and $\eta_2>0$, both tending to zero with $\varepsilon$,  such that 
$$
\MS_{X,T,d}(\eta_1,n)\leq \MS_{X,T,d'}(\varepsilon,n)\leq \MS_{X,T,d}(\eta_2,n)
$$
for every $n\in\N$.

We also could have defined this in terms of $\varepsilon$-$n$ separated sets: 
an \emph{$\varepsilon$-$n$ separated set} for $(X,T)$ is a finite set $\{x_1, \ldots, x_m\}$ 
of $X$ such that for all $x_i\neq x_j$ there exists $k\in\{0, \ldots, n-1\}$ such that $d(T^kx_i, T^kx_j)\geq\varepsilon$.  Taking $s_{X,T,d}(\varepsilon,n)$ to be the maximal cardinality 
of an $\varepsilon$-$n$ separated set in $X$, we obtain the same behavior  as  $n$ tends to infinity and $\varepsilon$ 
tends to $0$. This remark is used  (Section~\ref{subsec:uperbound}) to give an upper bound for the complexity of nilsystems.

The notion of topological complexity used here is closely related to the notion of the complexity of a cover studied in~\cite{BHM} and our results can be rephrased in this language (see Section~\ref{subsec:main}).

We show that every nilsystem $(X,T)$  has polynomial complexity (Theorem~\ref{th:main}) with an explicit degree, and most importantly, the
degree is the same both above and below. This places a constraint on any system having 
such a system as a factor. 
The upper bound  is related to  the well known fact that in a nilsystem, the orbits of two distinct points diverge at a polynomial rate. However, our bound is a global 
invariant, as opposed to such an infinitesimal characterization, and it describes the long time behavior. 
A polynomial upper bound was given in~\cite{DDMSY}, but without a clear control on the exponent.

A natural question is whether a weak converse of Theorem~\ref{th:main} holds:
\begin{question}
Let $(X,T)$ be a minimal topological dynamical system such that 
for every $\varepsilon > 0$, there exist constants $c(\varepsilon), c'(\varepsilon) > 0$ such that 
$$
c(\varepsilon)n\leq \MS_{X,T,d}(\varepsilon,n) \leq c'(\varepsilon)n
$$
for every $n\in\N$ and $c(\varepsilon) \to \infty$ as $\varepsilon \to 0$.
Is it true that  $(X,T)$ is a $2$-step nilsystem ?
\end{question}

The assumption on the growth of $c(\varepsilon)$ is needed in order to eliminate the possibility that $(X,T)$ is a subshift or, more generally, an expansive system.

\subsection{Nilsystems in symbolic dynamics}
We recall that a \emph{subshift over the alphabet $A$} is a closed, shift invariant subset $X$ of $A^\Z$, where $A$ is a finite set and $A^\Z$ is endowed with the natural compact topology and with the \emph{shift transformation}. 
Endowed with the restriction of the shift, a subshift is a particular type of topological dynamical system.

A natural question is to decide whether a given subshift has any nilfactor other than a rotation. Conversely, given a nilsystem, we can ask what kinds of subshifts 
admit it as a factor?  
This leads us to a notion of complexity classically used for subshifts:

\begin{definition}[see for example~\cite{F3}]
Let $(X,T)$ be a subshift on the finite alphabet $A$.
 For every integer $n\geq 1$, let $\MC_X(n)$ be the number of words of length $n$ occurring in $X$. 
The \emph{combinatorial complexity} of $(X,T)$ is the function $n\mapsto \MC_X(n)$. 
When there is no possible ambiguity, we omit the space from the notation and 
write $\MC(n)$.
\end{definition}

We show (Corollary~\ref{cor:main}) that a subshift with very low combinatorial complexity 
does not admit any nilfactor other than a rotation. More generally, subshifts with low combinatorial complexity do not admit any nilfactor of large order.

The classical Morse-Hedlund Theorem states that a subshift $X$ is finite (and thus only consists of periodic sequences) if and only if 
$\MC(n)=n$ for some $n$. On the other hand, \emph{Sturmian systems} satisfy $\MC(n)=n+1$ for every $n\geq 1$, and such systems are codings of irrational rotations on the circle.
This leads to the following question:
\begin{question}
Find an ``optimal'' coding of a minimal nilsystem $(X,T)$. 
More precisely, define a subshift $(Y,S)$ having $(X,T)$ as a nilfactor and with minimal possible complexity $\MC_Y$.
\end{question}

\subsection{Some questions}
\label{subsec:questions}
The discussion in Section~\ref{subsec:nilergo} leads naturally to other questions.
In the topological study of nilsystems (Section~\ref{subsec:itrotop}), 
notions of  complexity play a key  role in understanding the existence of nilfactors. We ask if something analogous holds in the ergodic setting:
\begin{question}
\label{qu:measure-complex}
Does the conclusion of Corollary~\ref{cor:spectrum} remain valid when the spectral hypothesis is replaced by  some hypothesis on 
the ``measure-theoretic complexity''~\cite{F4} or on the ``slow entropy''~\cite{KT} of the system?
\end{question}
A related question is:
\begin{question}
Compute the ``measure-theoretic complexity'' or the ``slow entropy'' of ergodic nilsystems.
\end{question}

The ergodic results of this paper deal with factors. A natural generalization consists in proving similar results for joinings:
\begin{conjecture*}
\label{conj:joining}
Let $(X,\mu,T)$ be an ergodic system satisfying the spectral hypothesis of Corollary~\ref{cor:spectrum} and let $(Y,\nu,S)$ be an ergodic nilsystem. Then every joining of these two systems is relatively independent with respect to the corresponding joining of their Kronecker factors.
\end{conjecture*}

We conclude by returning to our original motivation:
\begin{question}
Find ``natural'' nontrivial systems in any of the settings considered here (ergodic, 
topological or symbolic) that have an explicit nilfactor other than a rotation, 
meaning systems such that $Z_1(X)\neq Z_{2}(X)$.  
More generally, describe classes of systems with $Z_s(X)\neq Z_{s+1}(X)$ for 
some $s\geq 1$.  
\end{question}

\subsection*{Acknowledgments}
We thank Dave Morris, Terry Tao, and Jean-Paul Thouvenot for helpful 
discussions. 

\section{Notation and basic definitions}
\label{subsec:basic}
\subsection{Measure preserving systems}
Throughout the article, we omit the $\sigma$-algebra of measure preserving systems 
from our notation and write  $(X,\mu,T)$ instead of
$(X, \CB, \mu, T)$. All subsets of $X$  and functions on  $X$ are implicitly assumed to be measurable.  If $X$ is given a topological structure, the $\sigma$-algebra is assumed to be the Borel $\sigma$-algebra. 
For simplicity we always assume that the transformation $T$ is 
invertible. 
When $f$ is a function on $X$, we write, as usual, 
$Tf$ instead of $f\circ T$. 
Also, in a mild abuse of notation, we use $T$ to 
denote the unitary operator $f\mapsto f\circ T$ of $L^2(\mu)$.

\subsection{Nilsystems}
\label{subsec:nilsystems}
Let $G$ be a group. For $a,b\in G$, the commutator (making a conventional 
choice in the order) of these elements is defined to be 
$$
[a,b]=aba^{-1}b^{-1}.
$$
Throughout, we make use of several standard identities: for $a,b\in G$, we have 
$[b,a]=[a,b]^{-1}$ and for $a,b,c\in G$,
\begin{equation}\label{eq:com_abc}
 [a,bc]=[a,b]\,[b,[a,c]]\,[a,c].
\end{equation}
 If $a\in A$ and $A,B$ are subsets of $G$, we write 
$[a,B]$ for the group spanned by $\{[a,b]\colon b\in B\}$ and 
 $[A,B]$ for the group spanned by $\{[a,b]\colon a\in A,\ b\in B\}$.
The subgroups $G_j$, $j\geq 1$, of $G$ are defined inductively by
$$
G_1=G\ ; \ G_{j+1}=[G,G_j]\text{ for every }j\geq 1.
$$
One can check  that 
\begin{equation}
\label{eq:GkGl}
\text{for all }k,\ell\geq 1,\quad [G_k,G_\ell]\subset G_{k+\ell}.
\end{equation}

Let $s\geq 1$ be an integer. The group $G$ is \emph{$s$-step nilpotent} if $G_{s+1}=\{1_G\}$. In particular, $G$ is $1$-step nilpotent if and only if it is abelian. If $s\geq 2$ and $G$ is $s$-step nilpotent but not $(s-1)$-step nilpotent, then we have 
\begin{equation}
\label{eq:G2Gs}
G_1\supsetneq G_2\supsetneq G_3\supsetneq\dots\supsetneq G_{s-1}\supsetneq G_s\neq\{1_G\}.
\end{equation}

If $G$ is a Lie group, then $G_0$ denotes the connected component of its unit element $1_G$.

Let $s\geq 2$ be an integer, $G$ be an $s$-step nilpotent Lie group, and $\Gamma$ be 
a discrete and cocompact subgroup of $G$. The compact manifold $X=G/\Gamma$ is called an \emph{$s$-step nilmanifold}. The group $G$ acts on $X$ by left translation, and we write this action as $(g,x)\mapsto g\cdot x$. The \emph{Haar measure} $\mu$ on $X$ is the unique Borel probability measure $\mu$ on $X$ that is invariant under this action.

Let $\tau$ be a fixed element of $G$ and let $T\colon X\to X$ be the map $x\mapsto \tau\cdot x$. Then $(X,T)$ is called a  \emph{topological $s$-step nilsystem} and $(X,\mu,T)$ a 
\emph{measure theoretical $s$-step nilsystem}, or just an 
\emph{$s$-step nilsystem}. The basic  properties of nilsystems were established in~\cite{AGH} and~\cite{P}, and 
a more modern presentation is found in~\cite{L1}. In particular, we have the equivalences: $(X,T)$ is transitive if and only if it is minimal if and only if it is uniquely ergodic if and only if $(X,\mu,T)$ is ergodic.

If $(X,T)$ is minimal then, writing $G'$ for the subgroup of $G$ spanned by $G_0$ and $\tau$ and setting $\Gamma'=\Gamma\cap G'$, we have that $G=G'\Gamma$.  Thus $(X,T)\cong(X',T')$ where $X'=G'/\Gamma'$ and $T'$ is the translation by $\tau$ on $X'$. Therefore, 
without loss of generality we can restrict 
to the case that $G$ is spanned by $G_0$ and $\tau$.

We can also assume that $G_0$ is simply connected (see for example~\cite{M} or~\cite{AGH} for the case that $G=G_0$ and~\cite{L2} for the general case).

\section[Systems without nilfactors]{Measure theoretical results: systems without nilfactors}

\subsection{The spectrum of a nilsystem}
\label{subsec:spectral}

Improving a result of Green~\cite{AGH}, Stepin proved (see Starkov~\cite{Sta} 
for history and comments): 
\begin{theorem*}[\cite{St}]
\label{th:spectrum}
Let $(X=G/\Gamma,\mu,T)$ be an 
ergodic nilsystem with connected, simply connected group $G$. Then $L^2(\mu)$   can be written as the orthogonal sum $L^2(\mu)=\CH\oplus\CH'$ of two $T$-invariant subspaces.  
The space $\CH$ consists of functions $f\in L^2(\mu)$ that factorize through $G/G_2\Gamma$ and the restriction of $T$ to this space has discrete spectrum. The restriction of $T$ to $\CH'$ has Lebesgue spectrum of infinite multiplicity.
\end{theorem*}

Green,  Stepin,  and  Starkov only
 consider the case of a connected, simply connected group.
 But, in view of our applications, we need a similar result without any assumption 
 that the group be connected.  While the extension of Stepin's result to connected,  but not simply connected groups, is standard, the generalization for non-connected groups is 
 harder.  Nilsystems arising from non-connected groups can be quite different than those arising from connected ones, and there does not seem to be a direct method of deducing the general case from the particular one.  Adapting the existing proofs requires many changes, 
 and so instead of modifying existing proofs we give a different one.

For any $s> 2$, any $s$-step ergodic nilsystem that is not a rotation 
admits a $2$-step nilfactor that is not a rotation.  Thus, 
we only need such a spectral result for $2$-step nilsystems: 

\begin{proposition}
\label{prop:lebesgue}
Let $(X=G/\Gamma,\mu,T)$ be an 
ergodic $2$-step nilsystem that is not a rotation. Then $L^2(\mu)$ can be written as the orthogonal sum $L^2(\mu)=\CH\oplus\CH'$ of two closed $T$-invariant subspaces such that the restriction of $T$ to $\CH$ has discrete spectrum and its restriction to $\CH'$ has Lebesgue spectrum of infinite multiplicity.
\end{proposition}

The proof is elementary, but lengthy, and so we postpone it to
Appendix~\ref{sec:prooflebesgue}.

\begin{corollary}
\label{cor:spectrum}
Let $(X,\mu,T)$ be an ergodic system and assume that its spectrum does not admit a Lebesgue component with infinite multiplicity.
Then this system does not admit any nilsystem as a factor, other than a rotation factor.
\end{corollary}

This result applies, in particular, to
\begin{itemize}
\item 
Weakly mixing systems.
\item
Systems with singular maximal spectral type.
\item
Systems with finite spectral multiplicity. This class includes: 
\begin{itemize}
\item
Systems of finite rank, for example  substitution dynamical systems, linearly recurrent systems, 
Bratteli-Vershik systems of finite topological rank, and
interval exchange transformations (for definitions and references, see 
for example~\cite{ORW,Q,D,BDM,GJ}). 
\item 
Systems of local rank one  or of 
 funny rank one (see~\cite{F1,F2}, where the definitions are attributed to J.-P.~Thouvenot).
\end{itemize}
\end{itemize}

Since nilsystems have zero entropy, the result also applies to
\begin{itemize}
\item Systems whose Pinsker factor belongs to one of the preceding types.
\end{itemize}

\subsection{First application: lower bounds for multiple recurrence}
\label{subsec:appli1}

\begin{theorem}
\label{th:bornpol}
Assume that $(X,\mu,T)$ is an ergodic system satisfying $Z_s(X)=Z_1(X)$ for all $s>1$, 
for example a system satisfying one of the properties listed after Corollary~\ref{cor:spectrum}.
Let $p_1$, \dots, $p_k$ be integer polynomials satisfying $p_i(0)=0$ for $1\leq i\leq k$. 
Then for every $A\subset X$ and every $\varepsilon >0$, the set
\begin{equation}
\label{eq:progpolk}
\bigl\{n\in\N\colon
\mu(A\cap T^{-p_1(n)}A\cap T^{-p_2(n)}A\cap\ldots\cap T^{-p_k(n)}A) >\mu(A)^{k+1}-\varepsilon\bigr\}
\end{equation}
is syndetic.
\end{theorem}
\begin{remark}
Here and in Theorems~\ref{th:bornelin} and~\ref{th:weighted},
the hypothesis can be replaced by $Z_2(X)=Z_1(X)$; it is known that this condition  implies that $Z_s(X)=Z_1(X)$ for all $s>1$ (this follows, for example, from the inclusions~\eqref{eq:G2Gs} in Section~\ref{subsec:nilsystems}).
\end{remark}

In~\cite{BHK} it is showed that  the conclusion of Theorem~\ref{th:bornpol} does not hold for non-ergodic systems, even in the simple case of 
$k=2$, $p_1(n)=n$ and $p_2(n)=2n$. The conclusion also fails for general ergodic systems, for example for $k\geq 4$ and $p_j(n)=jn$ for $1\leq j\leq k$. In both of these cases, the set of integers defined by~\eqref{eq:progpolk} may be empty.

On the other hand, the conclusion of Theorem~\ref{th:bornpol} holds for weakly mixing systems~\cite{B}. Similar lower bounds for some particular choices of polynomials are found in~\cite{BHK,FrK,Fr1}.

For convenience, we begin the proof with the case of linear exponents and then explain how the method extends to the polynomial case. We first show:

\begin{theorem}
\label{th:bornelin}
Assume that  $(X, \mu, T)$ is an ergodic system with 
$Z_s(X)=Z_1(X)$ for all $s>1$, for example a system satisfying one of the properties listed after Corollary~\ref{cor:spectrum}.
Then
for any integer $k\geq 1$, any set $A\subset X$, and any $\varepsilon >0$, the set
\begin{equation}
\label{eq:progk}
\{n\in\N\colon\mu(A\cap T^{-n}A\cap \cdot\ldots\cdot\cap T^{-kn}A)> \mu(A)^{k+1}-\varepsilon\}
\end{equation}
is syndetic.
\end{theorem}

\begin{proof}
Fix $A\subset X$, an integer $k\geq 1$, and $\varepsilon>0$.

Set $g=\E(\one_A\mid Z_1(X))$ and notice that $0\leq g\leq 1$. Recall that the  Kronecker factor $(Z_1(X), \nu,T)$  of $(X, \mu,T)$ is a compact abelian group, endowed with its Haar measure $\nu$, Borel 
$\sigma$-algebra $\CZ_1$, and the transformation $T$ is translation by some element $\alpha\in Z_1$. For simplicity, in this proof we 
write $Z_1$ instead of $Z_1(X)$.  For each $t\in Z_1$, let $g_t$ be the function given by $g_t(x)=g(x+t)$. Then there exists a neighborhood $U$ of $0$ in $Z_1$ such that $\norm{g_t-g}_{L^1(\nu)}<\varepsilon/k^2$ for every $t\in U$. Therefore, for $t\in U$, we have that
$\norm{g_{jt}-g}_{L^1(\nu)}<j\varepsilon/k^2$ for every integer $1\leq j\leq k$ and
$$ 
\int g\cdot g_t\cdot \ldots\cdot g_{kt}\,d\nu>
\int g^{k+1}\,d\nu -
\sum_{j=1}^k \frac{j\varepsilon}{k^2}
\geq\mu(A)^{k+1}-\varepsilon .
$$
Let $\Lambda :=\{ n\in\N\colon n\alpha\in U\}$.
Then for every $n\in\Lambda$, 
$$
\int g\cdot T^ng\cdot\ldots\cdot T^{kn}g\,d\nu
>\mu(A)^{k+1}- \varepsilon .
$$
Furthermore,  $\Lambda$ is a syndetic set, and thus there exists an integer $L>0$ such that every interval of length $L$ in $\N$ contains at least one element of $\Lambda$.

On the other hand, 
by hypothesis, $g=\E(\one_A\mid Z_s(X))$ since $Z_s(X) = Z_1$ for all $s\geq 1$.  
Thus by~\cite[Corollary~4.6]{BHK}, the difference
$$
a_n:=\mu(A\cap T^{-n}A\cap \cdot\ldots\cdot\cap T^{-kn}A)-
\int g\cdot T^ng\cdot \dots\cdot T^{kn}g\,d\nu
$$
converges to $0$ in uniform density, meaning that
$$
\lim_{N\to+\infty}\sup_{M\in\N}\frac{1}{N}\sum_{n=M}^{M+N-1}|a_n| = 0.
$$
 In particular,  the set 
$$
\Lambda'
:=\Bigl\{ n\in\N\colon 
\mu(A\cap T^{-n}A\cap \ldots \cap T^{-kn}A)>
\int g\cdot T^ng\cdot \ldots\cdot T^{kn}g\,d\nu\,-\, \varepsilon \Bigr\}
$$
has lower Banach density one, meaning that 
$$
\lim_{N\to\infty}\inf_{M\in\N} \frac 1N\bigl|\Lambda'\cap [M,M+N)\bigr|=1.
$$
 Thus there 
exists an integer $N>0$ such that every interval of length $N$ in $\N$ contains $L$ consecutive elements of $\Lambda'$. 
By definition of $L$, any interval of length $N$ in $\N$ contains  some $n\in\Lambda'\cap\Lambda$ and this integer $n$ satisfies~\eqref{eq:progk}.
\end{proof}

Now we prove Theorem~\ref{th:bornpol}, that is, the extension of Theorem~\ref{th:bornelin} for polynomial iterates.  As the proof is similar but notationally more 
cumbersome, we only include an outline of the steps.

\begin{proof}[Proof of Theorem~\ref{th:bornpol}]
Assume that  $(X, \mu, T)$ is an ergodic system with 
 $Z_{s}(X)=Z_{1}(X)$ for all $s>1$ and that
 $p_1$, \dots, $p_k$ are integer polynomials satisfying $p_i(0)=0$ for $1\leq i\leq k$. 
In this proof, we write $Z_1,Z_s,\dots$, instead of $Z_1(X),Z_s(X),\dots$.

By~\cite{HK2} and~\cite{L2}, 
there exists an integer $s \geq 1$ such that for all functions $f_0,f_1,\dots,f_k\in L^\infty(\mu)$,   the averages over $[M_i,N_i)$ of 
\begin{multline*}
\int f_0\cdot T^{p_1(n)}f_1\cdot\ldots\cdot T^{p_k(n)}f_k\,d\mu\\
-
\int \E(f_0\mid Z_s)\cdot T^{p_1(n)}\E(f_1\mid Z_s)\cdot 
\cdot\ldots\cdot T^{p_k(n)}\E(f_k\mid Z_s)\,d\mu
\end{multline*}
converge to zero for  all sequences $(M_i)$  and $(N_i)$ of integers such that $N_i-M_i\to+\infty$.
Proceeding as in the proof of the deduction of Corollary~4.5 from Theorem~4.4 of~\cite{BHK}, 
we deduce that 
\begin{multline*}
\int f_0\cdot T^{p_1(n)}f_1\cdot\ldots\cdot T^{p_k(n)}f_k\,d\mu\\
-
\int \E(f_0\mid Z_{s+1})\cdot T^{p_1(n)}\E(f_1\mid Z_{s+1})\cdot\ldots\cdot T^{p_k(n)}\E(f_k\mid Z_{s+1})\,d\mu
\end{multline*}
converges to zero in uniform density.
By hypothesis and Corollary~\ref{cor:spectrum}, we have that $Z_{s+1}=Z_1$. Applying this with $f_0=f_1=\dots=f_k=\one_A$, and writing $g= \E(\one_A\mid Z_1)$, we conclude that
$$
\mu(A\cap T^{-p_1(n)}A\cap \dots\cap T^{-p_k(n)}A)- 
\int g\cdot T^{p_1(n)}g\cdot
\cdot\ldots\cdot T^{p_k(n)} g\,d\mu
$$
converges to zero in uniform density.

We continue as in the proof of Theorem~\ref{th:bornelin}. Let $\nu$ denote the Haar measure of $Z_1$ and let $\alpha$ be the element of $Z_1$ defining its transformation. 
For $t$ and $z\in Z_1$, write $g_t(z)=g(z+t)$; choose an open neighborhood $U$ of $0$ in $Z_1$ such that 
$\norm{g_t-g}_{L^1(\nu)}<\varepsilon/k^2$ for every $t\in U$. 
We now use a standard equidistribution method.  Let 
$H$ be the closed subgroup of $\T^k$ spanned by 
$\bigl(p_1(n)t, \ldots, p_k(n)t\bigr)$ for $n\in\Z$ and $t\in\T$.  We have that 
$H$ is equal to the set of $(t_1, \ldots, t_k)\in\T^k$ such that 
$a_1t_1+\ldots+a_kt_k = 0$ for every choice of $(a_1, \ldots, a_k)\in\Z^k$
such that 
$a_1p_1+\ldots+a_kp_k$ is identically zero.  
By Weyl's Theorem~\cite{We}, the sequence 
$\bigl(p_1(n)\alpha, \ldots, p_k(n)\alpha\bigr)$ is well distributed in 
$H$, meaning that 
for every continuous function $\phi$ on $\T^k$, 
$$
\frac{1}{N} \sum_{n=M}^{N+M-1}\phi\bigl((p_1(n)\alpha, \ldots, p_k(n)\alpha)\bigr)\longrightarrow
\int\phi\,dm_H, \text{ uniformly in M},
$$
where $m_H$ denotes the Haar measure on $H$.  We deduce that for the open set $U$,
$$
(p_1(n)\alpha, \ldots, p_k(n)\alpha)\in U\times\ldots\times U
$$
for a syndetic set of $n$.  We conclude as in the proof of Theorem~\ref{th:bornelin}.
\end{proof}

\subsection{Second application: weighted multiple averages}
\label{subsec:appli2}

The second application is a strengthening of results in~\cite{HK3} and~\cite{C} on the convergence of weighted polynomial multiple averages. Recall that the Kronecker factor of an ergodic system is naturally endowed with a topology, making it a compact abelian group. 

\begin{theorem}
\label{th:weighted}
Let $(X,T)$ be a uniquely ergodic topological dynamical  system with invariant measure $\mu$. 
Assume that $(X,\mu,T)$ satisfies $Z_{s}(X)=Z_{1}(X)$ for all $s>1$, for example a system satisfying one of the properties listed after Corollary~\ref{cor:spectrum},
and that the projection 
$\pi_1$ of $X$ onto its Kronecker factor  is continuous.
Then for any Riemann integrable function $\phi$ on $X$,  any $x\in X$,
any system $(Y,\nu,S)$, any $k\geq 1$, any functions $f_1, \ldots , f_k \in L^\infty(\nu)$, and any integer  polynomials $p_1, \ldots, p_k$, the averages
$$
\frac 1 N\sum_{n=0}^{N-1}
\phi(T^nx)\cdot S^{p_1(n)}f_1 \cdot S^{p_2(n)}f_2 \cdot\ldots\cdot  S^{p_k(n)}f_k
$$
converge in $L^2(\nu)$ as $N\to+\infty$.
\end{theorem}

In particular, this result applies for substitution dynamical systems~\cite{Q}, and more generally for many linearly recurrent systems and systems of finite topological rank~\cite{BDM}. In~\cite{HK3}, it was proved in the case of linear polynomials for particular sequences, 
including, for example the Thue-Morse sequence.

\begin{proof}
Assume that $(X,T)$ is a uniquely ergodic system satisfying the hypotheses of Theorem~\ref{th:weighted}.
Since for every $s\geq 1$ we have that $Z_s(X)=Z_1(X)$, by hypothesis the projection $\pi_s\colon X\to Z_s(X)$ is continuous.
Theorem~\ref{th:weighted} 
follows immediately by combining two results 
in the literature. The first is a weighted ergodic average for 
nilsequences:
\begin{theorem*}
[\cite{HK3}, Theorem 2.19 and Proposition 7.1]
Let $(X,T)$ be a uniquely ergodic system with invariant measure $\mu$ 
and let $m\geq 1$ be an integer. 
Assume that the factor map $\pi_{s}\colon X\to Z_{s}(X)$ is continuous.
Then for any Riemann integrable function $\phi$ on $X$,  $x\in X$, and
 $s$-step nilsequence $\ub=(b_n\colon n\in\Z)$, 
 the limit
$$
\lim_{N\to+\infty}\frac 1{N}\sum_{n=0}^{N-1} \phi(T^nx)b_n
$$
exists.
\end{theorem*}

The second is a weighted ergodic theorem for multiple convergence along integer
polynomials proved by Chu that generalized the linear case in~\cite{HK3}:
\begin{theorem*}[\cite{C}, Theorem 1.3]
For any $k, d \in \N$,
there exists an integer $s \geq 1$ with the following property: for any bounded sequence
$\ua = (a_n \colon n \in \Z)$, if the averages
$$
\frac 1 N\sum_{n=0}^{N-1}a_nb_n
$$
converge as $N \to+\infty$  for every $s$-step nilsequence $\ub = (b_n \colon n \in \Z)$, then for every
system $(Y,\nu,S)$, all $f_1, \ldots , f_k \in L^\infty(\nu)$, and all integer polynomials $p_1, \ldots, p_k$ of
degree $\leq d$, the averages
$$
\frac 1 N\sum_{n=0}^{N-1}
a_n\,S^{p_1(n)}f_1 \cdot S^{p_2(n)}f_2 \cdot\ldots\cdot  S^{p_k(n)}f_k
$$
converge in $L^2(\nu)$.
\end{theorem*}
\end{proof}

\section{Complexity of topological nilsystems}

\subsection{Complexity and commutator dimension}
\label{subsec:main}

Before stating the theorem, we define:
\begin{definition}
\label{def:com_dim}
If $(X=G/\Gamma, T)$ is an $s$-step nilsystem for some $s\geq 1$ and if $\tau\in G$ 
is the element defining $T$, the \emph{total commutator dimension} $p$ of $X$ 
is defined to be 
 \begin{equation}
\label{eq:defp}
p=\sum_{\ell=1}^{s-1}\dim\bigl(\range\bigl(\ad_\tau-\id)^\ell\bigr)\bigr).
\end{equation}
\end{definition}

Our main result is:
\begin{theorem}
\label{th:main}
Let $(X=G/\Gamma, T)$ be a minimal $s$-step  nilsystem for some $s\geq 2$ and
 assume that $(X,T)$ is not an $(s-1)$-step nilsystem.
Let $d_X$ be a distance on $X$ defining its topology.  
Then
  for every $\varepsilon>0$ that is sufficiently small, there exist positive constants $C(\varepsilon)$ and $C'(\varepsilon)$ such that the  topological complexity $\CS(\varepsilon,n)$ of $(X,T)$ for the distance $d_X$ satisfies
\begin{equation}
\label{eq:main}
C(\varepsilon)n^p\leq\CS(\varepsilon,n)\leq C'(\varepsilon)n^p\text{ for every }n\geq 1,
\end{equation}
where $p$ is the total commutator dimension of $X$. Moreover, $p\geq s-1$ and 
$C(\varepsilon)\to+\infty$ when $\varepsilon\to 0$.  

Furthermore, for a suitably chosen distance on $X$, one can take
$C(\varepsilon)=C\varepsilon^{-d}$ and $C'(\varepsilon)=C'\varepsilon^{-d}$, where $C$ and $C'$ are positive constants and $d$ is the dimension of $X$.
\end{theorem}

The result can be translated into the language of complexity of covers studied in~\cite{BHM}.
 We start by reviewing the definition. Let $\CU$ be an open cover
of $X$, and for every integer $n\in\N$, write
$$
\CU_0^{n-1}:=\bigvee_{j=0}^{n-1} T^{-j}\CU.
$$
 Define $\MN(\CU,n)$ to be the minimal cardinality of a subcover of $\CU_0^{n-1}$; the \emph{complexity function} of $\CU$ is the map $n\mapsto \MN(\CU,n)$.
 We have that $\MS(\varepsilon,n)\leq c(\varepsilon)$ for every $\varepsilon>0$ is 
equivalent to $\MN(\CU,n)\leq C'(\CU)$ for every open cover $\CU$ of $X$, and these conditions 
are equivalent to the system being a rotation.
The upper bound in~\eqref{eq:main} can be rephrased as saying that for every open cover $\CU$ of $X$, there exists a constant $C'(\CU)$ with $\MN(\CU,n)\leq C'(\CU)n^p$ for every $n\in\N$. The lower bound means that there exists an open cover $\CU$ and a constant $C(\CU)$ such that $\MN(\CU,n)\geq C(\CU)n^p$ for every $n\in\N$.

\begin{corollary}
\label{cor:main}
Let $s\geq 1$ and let
 $(X=G/\Gamma, T)$ be a  minimal nilsystem  that is not an  $s$-step nilsystem. Then
$$
\liminf_{n\to+\infty}\frac 1{n^s}\CS(\varepsilon,n)\to +\infty\text{ as }\varepsilon\to 0.
$$
\end{corollary}

Applications of these results are given in Section~\ref{sec:top_app}.

\subsection{Conventions and notation}
In the sequel, $s\geq 2$ is an integer and $(X=G/\Gamma,T)$ is a minimal $s$-step nilsystem that is not an $(s-1)$-step nilsystem. We let $\tau$ denote the element of $G$ defining the transformation $T$.

As explained in Section~\ref{subsec:nilsystems}, the assumption of minimality allows us to   assume that $G$ is  spanned by the connected component $G_0$ of the unit element $1_G$ and $\tau$.
We make further assumptions on the choice of the presentation $G/\Gamma$ of 
the nilsystem $X$ and  then prove Theorem~\ref{th:main} 
under these additional assumptions. This clearly implies the result in the general case.

Let $\mu$ denote the Haar measure of $X$ and let $\lambda$ denote 
the Haar measure of $G$.  We normalize 
$\lambda$ such that the measure of any (Borel) fundamental domain of the projection $\pi\colon G\to X$ is equal to $1$.

Here, and  again in Section~\ref{subsec:lie}, we impose conditions on the distance $d_X$ on $X$ defining the topology on $X$.  Again, the conclusions of Theorem~\ref{th:main} remain valid for a general distance defining the topology.

Throughout the proof, we often fix some $\varepsilon>0$ and assume  that $\varepsilon$ is sufficiently small. This means that it is smaller than some constant depending only on the nilsystem $(X=G/\Gamma,T)$ and the distance $d_{X}$ defined on it, and not on any other parameter such as the integer $n$.

Finally, we choose a bounded Borel fundamental domain $D$ of the projection $\pi\colon G\to X$.

\subsection{Some preliminaries}
\label{subsec:prelim}
\subsubsection{Choosing a distance on $X$}
First, we choose a distance $d_G$ on the group $G$ that defines  its topology. For the moment, we only assume that this distance is invariant under right translations, meaning that
for all $g,g',h\in G$, 
$$
d_G(gh,g'h)=d_G(g,g').
$$

The nilmanifold $X$ is endowed with the quotient distance, meaning that 
for $x,y\in X$,
\begin{equation}
\label{eq:defdX}
d_X(x,y)=\inf\bigl\{d_G(g,h)\colon \pi(g)=x\text{ and }
\pi(h)=y\}.
\end{equation}
In other words, for all $g,h\in G$, we have
\begin{equation}
\label{eq:defdX}
d_X\bigl(\pi(g),\pi(h)\bigr)=\inf_{\alpha,\beta\in\Gamma}d_G(g\alpha,h\beta)
=\inf_{\gamma\in\Gamma}d_G(g,h\gamma).
\end{equation}
Since the inverse image under $\pi$ of every point of $X$ is discrete, 
the infimums in these last two formulas are attained.

We recall that for $2\leq j\leq s$, $G_j$ and $G_j\Gamma$ are  closed subgroups of $G$ and that $\Gamma\cap G_j$ is a cocompact subgroup of $G_j$.
In particular, we deduce that there exists $\delta_0>0$ such that for $2\leq j\leq s$,
\begin{equation}
\label{eq:dGj}
\text{if }\gamma\in\Gamma\text{ is within a distance }\delta_0\text{ of }G_j\text{, then }\gamma\in G_j\cap\Gamma.
\end{equation}
In particular, we deduce that
\begin{equation}
\label{eq:dggprime}
\text{if }d_G(g,g')<\delta_0\text{ and }\pi(g)=\pi(g'),\text{ then }g=g'.
\end{equation}

\subsubsection{Commutators}

The following lemma is used to prove that the exponent $p$ of Theorem~\ref{th:main} is $\geq s-1$.

\begin{lemma} 
\label{lem:tau}
For $1\leq \ell \leq s-1$, we have that $[\tau,G_\ell]\not\subset G_{\ell+2}$.
\end{lemma}

\begin{proof}
Let $G'=G/G_{\ell+2}$, $\pi\colon G\to G'$ be the quotient homomorphism, $\Gamma'=\pi(\Gamma)$ and $\tau'=\pi(\tau)$. Let $T'$ be the translation by $\tau'$ on $X'=G'/\Gamma'$.
Then for every $j\geq 1$, we have that $\pi(G_j)=G'_j$.
Therefore $G'_{\ell+2}=\{ 1_{G'}\}$ and $(X',T')$ is an $(\ell+1)$-step nilsystem. But  $G_{\ell+1}\neq G_{\ell+2}$ and so $G'_{\ell+1}\neq\{1_{G'}\}$.  Thus 
 $(X',T')$ is not a $\ell$-step nilsystem.
Moreover, since $(X',T')$ is  a factor of $(X,T)$, it is minimal. 

We remark that   
$\pi([\tau,G_\ell])= [\tau',G'_\ell]$. 
Therefore, substituting $X'$ for $X$, we are reduced to showing that $\tau$ does not commute with $G_\ell$ when $X$ is a $(\ell+1)$-step, but not $\ell$-step, nilsystem.

Assume that $\tau$ commutes with $G_\ell$.
Consider the commutator map $G_{\ell}\times G_0\to G_{\ell+1}$. Since $G$ is $(\ell+1)$-step nilpotent, by~\eqref{eq:com_abc} this map is multiplicative in each coordinate 
separately.
But by~\eqref{eq:GkGl}, $[G_{\ell},G_2]$ is trivial and so 
this commutator map induces a continuous map $G_{\ell}\times G_0/G_2\to G_{\ell+1}$. Finally, $G_{\ell}$ and $G_0/G_2$ are abelian and so this continuous map is bilinear.

On the other hand, by minimality the subgroup of $G$ spanned by $\Gamma$ and $\tau$ is dense in $G$. This and the hypothesis imply that
 $[G_{\ell}\cap\Gamma,G]=[G_\ell\cap\Gamma,\Gamma]\subset \Gamma$.
Therefore, for $\gamma\in G_{\ell}\cap \Gamma$,  the map 
$g\mapsto [\gamma,g]$ continuously maps  the connected group $G_0$ to the discrete group 
$\Gamma$, and so this map is trivial. We conclude that
$G_{\ell}\cap\Gamma$ commutes with $G_0$.

Therefore, the commutator map $G_{\ell}\times G_0\to G_{\ell+1}$ 
induces a bilinear continuous map 
\begin{equation}
\label{eq:com}
\psi\colon \frac{G_\ell}{G_{\ell}\cap\Gamma}\times\frac {G_0}{G_2} \to G_{\ell+1}.
\end{equation}
Let $\chi\colon G_{\ell+1}\to\T$ be a character of the abelian group $G_{\ell+1}$. Then 
$\chi\circ\psi\colon G_{\ell}/(G_{\ell}\cap\Gamma)\times G_0/G_2\to\T$ induces a continuous group homomorphism from group $G_0/G_2$ to the dual group of the compact abelian group 
$G_{\ell}/(G_{\ell}\cap\Gamma)$. Since this dual group is discrete and  $G_0/G_2$ is connected, this group homomorphism is trivial. It follows that $\chi\circ\psi$ is the trivial map.

As this holds for every character $\chi$ of $G_{\ell+1}$, $\psi$ is the trivial map. 
Combining this with definition~\eqref{eq:com}, it follows that the commutator map $G_{\ell}\times G_0\to G_{\ell+1}$ is trivial, and so $G_{\ell}$ commutes with $G_0$.

By assumption, $G$ is spanned by $G_0$ and $\tau$ and $G_{\ell}$ commutes with $\tau$. 
It follows that $G_{\ell}$ is included in the center of $G$, that $G_{\ell+1}$ is trivial, and thus that $G$ is $\ell$-step nilpotent, a contradiction.
\end{proof}

\subsubsection{Some linear algebra}
We make use of the following estimate from linear algebra. 
The proof is postponed to Appendix~\ref{ap:linear}.
\begin{proposition}
\label{prop:lineaire}
Let $\R^d$ be endowed with the Euclidean norm $\norm\cdot$ and let
the Lebesgue measure of a Borel subset $K$ of $\R^d$ be written $|K|$. 
Let $A$ be a $d\times d$ matrix and assume that it is unipotent, meaning that $(\id-A)^d=0$. For every integer $n\geq 2$, let 
$$
\CW_n=\bigl\{ \xi\in\R^d\colon \norm{A^k\xi}\leq 1\text{ for }1\leq k<n\bigr\}.
$$
If
$$
p=\sum_{k=1}^{d-1}\dim(\range(\id-A)^k), 
$$
there exist positive constants $C$ and $C'$ (depending on $d$ and on $A$) such that
\begin{equation}
\label{eq:boundmeasure}
C n^{-p} \leq |\CW_n|\leq C'n^{-p}
\end{equation}
for every $n$.
\end{proposition}

\subsection{Reduction to a local problem}
\label{subsec:cross}

Throughout this section, $n>1$ is an integer. We assume that $\varepsilon$ is given and is sufficiently small, 
$x_0$ is a point in $X$, and $h\in D$ is chosen such that $\pi(h)=x_0$.

For the moment, we view $x_0$ and $h$ as fixed, but it is important that the bounds do not depend on $x_0\in X$, and thus also not on $h\in D$.

We  study the set $V\subset X$ defined by 
$$
V = \{x\in X \colon d_X(T^kx_0,T^kx)<\varepsilon\text{ for }0\leq k<n\}.
$$

Since the infimum in the definition of $d_X$ (see~\eqref{eq:defdX}) 
is attained, for every $x\in V$ there exists $g\in G$ with 
\begin{equation}
\label{eq:gonx}
x=g\cdot x_0\text{ and }d_G(g,1_G)=d_X(x,x_0).
\end{equation}
Furthermore, for sufficiently small $\varepsilon$, it follows from~\eqref{eq:dggprime} that these conditions completely define $g$.  Let $W$ denote the set of elements $g$ associated in this way to points of $V$.  Again, the choice of small $\varepsilon$ implies that $W$ is included in the connected component $G_0$ of $1_G$ in $G$.

\begin{convention}
In the sequel, we often view $x$, and so $g$, fixed. When needed, we emphasize the dependence of $g$ on $x$ satisfying~\eqref{eq:gonx} by writing $g(x)$ instead of $g$.
\end{convention}

By hypothesis and~\eqref{eq:defdX}, for $0\leq k<n$, there exists $\gamma_k\in\Gamma$ with
\begin{equation}
\label{eq:defgammak}
d_G(\tau^kh\gamma_k,\tau^kgh)=d_X(\tau^k\cdot x_0,\tau^kg\cdot x_0)=d_X(T^kx_0,T^k x)<\varepsilon
\end{equation}
and, by~\eqref{eq:dggprime} again, this $\gamma_k$ is unique. Clearly, $\gamma_0=1_G$.

In  the remainder of this section, we show: 
\begin{lemma}
\label{lem:gammak}
Let $x_0$, $x$, $g$, and $\gamma_k$ be as above.
 Then
$\gamma_k=1_G$ for $0\leq k<n$.
\end{lemma}

\subsubsection{Initial computations}

From the characterization~\eqref{eq:defgammak} of $\gamma_k$, it follows that for $0\leq k<n$,
\begin{align*}
\varepsilon
& > d_G\bigl([\tau^kh,\gamma_k] \gamma_k\tau^kh\;,\; \tau^kgh\bigr)\\
&=
d_G\bigl([\tau^k h,\gamma_k]\gamma_k\tau^k\;,\;\tau^kg\bigr)
\text{ by right invariance}\\
&=d_G\bigl( [\tau^kh,\gamma_k]\gamma_k\tau^k\;,\; g[g^{-1},\tau^k]\tau^k\bigr)\\
&=d_G\bigl( [\tau^kh,\gamma_k]\gamma_k\;,\; g[g^{-1},\tau^k]\bigr)
\text{ by right invariance (again).}
\end{align*}

Since $d_G(g,1_G)<\varepsilon$, we have that
\begin{equation}
\label{eq:bound}
d_G\bigl( [\tau^kh,\gamma_k]\gamma_k\;,\; [g^{-1},\tau^k]\bigr)
<2\varepsilon \text{ for }0\leq k<n.
\end{equation}

We proceed by  induction on the degree $s$ of the $s$-step nilsystem $X$.

 \subsubsection{The case $s=2$}
 
 Since $G_2$ is included to the center of $G$, it follows from~\eqref{eq:com_abc} that the map $(x,y)\mapsto[x,y]$ from $G\times G$ to $G_2$ is bilinear and thus
 $[g^{-1},\tau^k]=[g,\tau]^{-k}=[\tau,g]^k$ for every $k$. Moreover, 
$[\tau^kh,\gamma_k]$ belongs to the center of $G$, and Equation~\eqref{eq:bound} can be rewritten as  
\begin{equation}
\label{eq:bound2}
d_G\bigl( \gamma_k\;,\; [\tau^kh,\gamma_k]^{-1}[\tau,g]^k\bigr)
<2\varepsilon
\text{ for }0\leq k<n.
\end{equation}
Therefore, for $0\leq k<n$, we have that $\gamma_k$ is within distance $2\varepsilon$ of $G_2$. By Remark~\eqref{eq:dGj}, we have that 
$$
\gamma_k\in G_2\cap\Gamma,
$$
 and in particular $\gamma_k$ belongs to the center of $G$.
Then~\eqref{eq:bound2} implies that 
\begin{equation}
\label{eq:bound2b}
d_G(\gamma_k, [\tau,g]^k)<2\varepsilon\text{ for }0\leq k<n.
\end{equation}

Since $d_G(g,1_G)<\varepsilon$, it follows that $d_G([\tau,g],1_G)<C\varepsilon$ for some $C>0$.
Thus $d_G(\gamma_1,1_G)<(C+2)\varepsilon$. Since $\Gamma$ is discrete and $\varepsilon$ is small, 
$$
\gamma_0=\gamma_1=1_G.
$$
On the other hand, by~\eqref{eq:bound2b} again,  
$$
d_G\bigl(\gamma_k\gamma_{k+1}^{-2}\gamma_{k+2}\;,\;1_G)< 8\varepsilon  
\text{ for }0\leq k<n-2.
$$
Since $\Gamma$ is discrete and $\varepsilon$ is small, 
$$\gamma_k\gamma_{k+1}^{-2}\gamma_{k+2}=1_G\text{ for }0\leq k<n-2.
$$
 We deduce that $\gamma_k=1_G$ for $0\leq k<n$, as announced.

\subsubsection{The general case}
Assume that $X$ is an $s$-step nilsystem for some $s> 2$ 
and that the result has been proven for an $(s-1)$-step nilsystem.  
Maintaining the same notation and conventions as above, we 
show that $\gamma_k=1$ for every $k$.

Taking the quotient by $G_s$, the induction hypothesis implies that  
$$
\gamma_k\in G_s\text{ for } 1\leq k<n
$$
  and in particular  $\gamma_k$ belongs to the center of $G$.
For every $k\in\Z$, we write
\begin{equation}
\label{eq:defgk}
g(k):= [g^{-1},\tau^k].
\end{equation}
By the estimate in Equation~\eqref{eq:bound}, we have that
\begin{equation}
\label{eq:bounds}
d_G\bigl( \gamma_k\;,\; g(k)\bigr)
<2\varepsilon \text{ for }0\leq k<n.\end{equation}

A new difficulty arises here that does not come up for $s=2$, as 
in general $g(k)$ does not belong to the center of $G$.
We begin by showing that  the map $k\mapsto g(k)$ 
 is a polynomial 
map from $\Z$ to $G_2$ (see~\cite{L1}).
The computations for this are fairly explicit and we include them.

First, since $d_G(g,1_G)<\varepsilon$, there exists a constant $C>0$ such that
$d_G\bigl([g^{-1},\tau^k],1_G\bigr)<C\varepsilon$ for 
$0\leq k\leq s-1$
and, by definition~\eqref{eq:defgk} of $g(k)$ and~\eqref{eq:bounds},
$d_G(\gamma_k,1_G)<(C+2)\varepsilon$  for 
$0\leq k\leq s-1$.
Since $\Gamma$ is discrete and $\varepsilon$ is small, we deduce that 
\begin{equation}
\label{eq:gamma1gammas}
\gamma_k=1_G\text{ for }0\leq k\leq s-1.
\end{equation}

On the other hand,
recall that the  \emph{difference operator}  $D$ maps a sequence $(h(k)\colon k\geq 0)$ with values in $G$ to the sequence $(Dh(k)\colon k\geq 0)$ given by 
$$
Dh(k)=h(k)^{-1}h(k+1)\text{ for every } k\geq 0.
$$
Applying this definition to the sequence $\bigl(g(k) \colon k\geq 0\bigr)$ defined by~\eqref{eq:defgk},   it is easy to check by induction that for every $j\geq 1$,
$$
D^jg(k)=\tau^kv_j\tau^{-k}, 
$$
 where 
$$
v_1=[g^{-1},\tau]\text{ and }v_{j+1}=[v_j^{-1},\tau].
$$
Therefore, for every $j\geq 1$ we have $v_j\in G_{j+1}$ and $D^jg(k)\in G_{j+1}$ for every $k$.
(In the vocabulary of Leibman~\cite{L1}, this means that the sequence $\bigl(g(k)\colon k\geq 0\bigr)$ belongs to the class
${\mathscr P}_{(0,1,2,\dots)}G$.)
In particular  $v_s=1_G$ and 
$$
D^sg(k)=1_G\text{ for every }k\geq 0.
$$
By the definition of the difference operator and by induction, we have  that the sequence $(g(k)\colon k\geq 0)$ satisfies a recurrence relation of the form: for every $k\geq 0$, 
$$
1_G=D^sg(k)=g(k+m_1)^{\eta_1}g(k+m_2)^{\eta_2}\dots
 g(k+m_{2^s})^{\eta_{2^s}},
$$
where  
\begin{equation}
\label{eq:etasms}
 0\leq m_j<s \text{ and } \eta_j=\pm 1 \text{ for }1\leq j<2^s\ ;\
m_{2^s}=s \text{ and }\eta_{2^s}=1.
\end{equation}

Since $\gamma_\ell$ belongs to the center of $G$ for every $\ell$, 
\begin{multline*}
d_G\Bigl(1_G\;,\;\gamma_{k+m_1}^{\eta_1}\dots
 \gamma_{k+m_{2^s}}^{\eta_{2^s}}\Bigr)\\
= d_G\Bigl(
 \bigl(g(k+m_1)^{\eta_1}\dots
 g(k+m_{2^s})^{\eta_{2^s}}
\bigr)\bigl(\gamma_{k+m_1}^{\eta_1}\dots
 \gamma_{k+m_{2^s}}^{\eta_{2^s}}\bigr)^{-1} \;,\;1_G)\\
=d_G\Bigl( (g(k+m_1)\gamma_{k+m_1}^{-1})^{\eta_1} 
 \dots
(g(k+m_{2^s})\gamma_{k+m_{2^s}}^{-1})^{\eta_{2^s}} 
\;,\;1_G\Bigr)\\
\leq\sum_{j=1}^{2^s}d_G\bigl(g(k+m_j)\gamma_{k+m_j}^{-1}\;,\;1_G\bigr)
<2^{s+1}\varepsilon\quad \text{ by~\eqref{eq:bounds}.}
\end{multline*}

Since $\Gamma$ is  discrete and $\varepsilon$ is small, we deduce that
\begin{equation}
\label{eq:recurs}
\gamma_{k+m_1}^{\eta_1}\dots
 \gamma_{k+m_{2^s}}^{\eta_{2^s}}=1_G\text{ for every }k\geq 0.
\end{equation}
In other words, the $s^{\text{th}}$ iterated difference $(D^s\gamma_k)$ of the sequence $(\gamma_k\colon k\geq 0)$ is  trivial, and this sequence is a polynomial sequence in the abelian group 
$G_s\cap\Gamma$.

Recalling that $m_{2^s}=s$ and  $\eta_{2^s}=1$, 
combining~\eqref{eq:gamma1gammas}, \eqref{eq:recurs} and~\eqref{eq:etasms} and using induction, 
we have that $\gamma_k=1_G$ for  $0\leq k<n$ (see also
Proposition 3.1 in~\cite{L1}).

This concludes the proof of Lemma~\ref{lem:gammak}. \qed

\subsubsection{A summary} 
\label{subsec:summarize}
\begin{corollary}
\label{cor:summarize}
Let $x_0\in X$ and $h\in D$ with $\pi(h)=x_0$. 
Define 
$$
V=\{ x\in X\colon d_X(T^kx_0,T^kx)<\varepsilon\text{ for }0\leq k<n\}.
$$
Then, for every $x\in V$,  there exists a unique $g=g(x)\in  G$ satisfying 
$$
x=g\cdot x_0\text{ and }d_G(g,1)=d_X(x,x_0).
$$
We have
\begin{equation}
\label{eq:conjug}
d_G(\tau^k g\tau^{-k},1_G)
<\varepsilon\text{ for }0\leq k<n.
\end{equation}
\end{corollary}
The last formula follows from~\eqref{eq:defgammak}, Lemma~\ref{lem:gammak} and invariance of the distance under right translations.

\section{Proof of Theorem~\ref{th:main}: computing the complexity}
\subsection{Working in the Lie algebra}
\label{subsec:lie}
As discussed in Section~\ref{subsec:nilsystems},
we can  assume that the connected component $G_0$ of  the identity $1_G$
 is simply connected. 
 
 As we consider neighborhoods $W$ of $1_G$ such that $d_G(g,1_G)<\varepsilon$ for every $g\in W$, we have that $W\subset G_0$.  Thus it suffices 
 to give a description of $G_0$.

Recall that the exponential map is a homeomorphism from the Lie algebra $\CG$ onto $G_0$. For every $j\geq 1$, we denote the  Lie algebra of $G_j$  by 
$\CG_j$ and the exponential map takes $\CG_j$ onto $G_j$.

Let $\CG$ be endowed with some Euclidean norm $\norm\cdot$ and some orthonormal basis for this norm.  We use this basis
to identify $\CG$ and $\R^d$.

Let the distance  $d_G$ on $G$ be the right invariant Riemannian distance  such that the associated norm on the Lie algebra $\CG$ is $\norm\cdot$. We have
\begin{equation}
\label{eq:isometry}
d_G(\exp(\xi), 1_G)=\norm\xi\text{ for every }\xi\in\CG.
\end{equation}

As before, we denote the Lebesgue measure of a subset $\CL$ of $\CG$ by $|\CL|$.
In general, the exponential map does not take the Lebesgue measure of $\CG$ to the Haar measure $\lambda$  of $G$. But, because all elements $g$ under consideration belong to the compact subset 
$$
K:=\{g\in G\colon d_G(g,1_G)\leq 1\}
$$
of $G$, the density (with respect to $\lambda$) of the image of Lebesgue measure under the exponential map is bounded above and below by some positive constants $C$ and $C'$: 
\begin{multline}
\label{eq:comparemeasures}
\text{for every }L\subset K,\\
C\,\bigl|\bigl\{ \xi\in \CG\colon \exp(\xi)\in L\bigr\}\bigr|
\leq 
\lambda(L)
\leq
C'\,\bigl|\bigl\{ \xi\in \CG\colon \exp(\xi)\in L\bigr\}\bigr|. 
\end{multline}

We recall that the  linear map $\ad_\tau\colon\CG\to\CG$ is defined to be the differential
  of the map $g\mapsto \tau g\tau^{-1}$ from $G$ to itself, evaluated at the point $1_G$.

We have that
\begin{equation}
\label{eq:defAd}
\tau \exp(\xi)\tau^{-1} =\exp(\ad_\tau\xi)\text{ for every }\xi\in\CG.
\end{equation}
\begin{notation}
 We write $\Phi_\tau\colon G\to G$ for the map $g\mapsto 
\tau g\tau^{-1} g^{-1}$. 
\end{notation}
The differential of $\Phi_\tau$ evaluated at $1_G$ is 
$\ad_\tau-\id\colon\CG\to\CG$.
Since $G$ is $s$-step nilpotent, the $s^{\text{th}}$ iterate of $\Phi_\tau$ is the constant map $1_G$.  Thus 
$$
(\ad_\tau-\id)^s=0, 
$$
meaning that the  map $\ad_\tau-\id$ is nilpotent. More precisely, for every $k\geq 1$, we have that $\Phi_\tau(G_k)\subset G_{k+1}$ and thus 
$(\ad_\tau-\id)\CG_k\subset\CG_{k+1}$. 
%
\begin{lemma}
\label{lem:ps}
The total commutator dimension $p$ (Definition~\ref{def:com_dim}) of $X$ 
satisfies 
$$p\geq \dim\bigl(\range(\ad_\tau-\id)\bigr)\geq s-1.
$$
\end{lemma}

\begin{proof}
For $1\leq\ell\leq s-1$, the restriction of the map $\Phi_\tau$ to $G_\ell$ maps this group  to $G_{\ell+1}$.  Composing it with the factor map $G_{\ell+1}\to G_{\ell+1}/G_{\ell+2}$ we obtain a map $\Psi_\ell\colon
G_\ell\to G_{\ell+1}/G_{\ell+2}$. By~\eqref{eq:com_abc}, this map is a group homomorphism.
By Lemma~\ref{lem:tau}, this homomorphism is not trivial.

Now we take differentials  at the unit element $1_G$ of $G$. 
The differential of $\Phi_\tau$ at $1_G$  is $\ad_\tau-\id$, the differential at $1_G$ of the quotient map $G_{\ell+1}\to G_{\ell+1}/G_{\ell+2}$ is the quotient map
$p_\ell\colon \CG_{\ell+1}\to \CG_{\ell+1}/\CG_{\ell+2}$, and thus the differential of $\Psi_\ell$ at this point is $p_\ell\circ(\ad_\tau-\id)$.  
Since the group homomorphism $\Psi_\ell$ is not trivial, its differential at $1_G$ is not zero and thus 
$p_\ell\circ(\ad_\tau-\id)$ is not trivial.

We conclude that $(\ad_\tau-\id)\CG_\ell\not\subset \CG_{\ell+2}$. Since 
$(\ad_\tau-\id)\CG_{\ell+1}\subset \CG_{\ell+2}$, it follows that 
$(\ad_\tau-\id)\CG_\ell\neq(\ad_\tau-\id)\CG_{\ell+1}$.  Thus
$$
\dim\bigl((\ad_\tau-\id)\CG_\ell\bigr)\geq\dim\bigl((\ad_\tau-\id)\CG_{\ell+1}\bigr)+1.
$$
Summing this inequality for $1\leq\ell\leq s-1$, we conclude that
$$\dim\bigl((\ad_\tau-\id)\CG\bigr)\geq s-1.$$
\end{proof}

\subsection{Bounding the complexity}

\begin{notation}
Set  $$
\CK:=\{\xi\in\CG\colon \exp(\xi)\in K\}, 
$$
and note that $\CK$ is a compact subset of $\CG$.
Set
\begin{align*}
\CW_{n}
&:= \bigl\{ \xi\in\CG\colon \norm{\ad_\tau^k\xi}\leq 1\text{ for }0\leq k<n\bigr\}\ ;\\
W_{\varepsilon,n}
& :=\bigl\{\exp(\xi)\colon \xi\in \varepsilon\CW_{n}\bigr\}
\\
&= \bigl\{g\in G\colon d_G(\tau^k g\tau^{-k},1_G)<\varepsilon \text{ for }0\leq k<n\bigr\}
\text{ by~\eqref{eq:defAd} and~\eqref{eq:isometry}.}
\end{align*}
\end{notation}
(In the last line, we can write $g\in G$ instead of $G_0$, as $\varepsilon$ small implies 
that all points of $G$ at a distance $<\varepsilon$ of $1_G$ belong to $G_0$.)

Since the matrix of $\ad_\tau$ is unipotent, by definition~\eqref{eq:defp} of $p$  and Proposition~\ref{prop:lineaire}, we have that
$C\varepsilon^dn^{-p}\leq |\varepsilon\CW_n|\leq C'\varepsilon^dn^{-p}$. 
Thus, since $W_{\varepsilon,n}\subset K$,
by~\eqref{eq:comparemeasures} it follows that:

\begin{corollary}
\label{cor:measuW}
For every  sufficiently small $\varepsilon> 0$  and every $n\in\N$, 
\begin{equation}
C\varepsilon^dn^{-p}\leq \lambda(W_{\varepsilon,n})\leq C'\varepsilon^dn^{-p}.
\end{equation}
\end{corollary}

\subsubsection{Lower bound for the complexity}
We use this description to prove the lower bound of Theorem~\ref{th:main}.

Assume that $\varepsilon> 0$ is sufficiently small and that $\{x_1,\dots,x_N\}$ is an 
$\varepsilon$-$n$ spanning set for $X$.
Let 
$$
V_j=\{x\in X\colon d_X(T^k x_j,T^k x)<\varepsilon\text{ for }0\leq j<n\}.
$$
By hypothesis,  the union of the sets $V_j$ cover $X$.

For $1\leq j\leq N$, choose $h_j\in D$  such that 
$\pi(h_j)=x_j$.
We are in the setting of Corollary~\ref{cor:summarize}.
For every $j\geq 1$ and every $x\in V_j$, there exists a unique $g_j(x)\in G$ with $x=g_j(x)\cdot x_j=\pi(g_j(x)h_j)$ and 
$d_G(g_j(x),1_G)=d_X(x_j,x)<\varepsilon$.
By~\eqref{eq:conjug}, we have that $d_G(\tau^k g_j(x)\tau^{-k}, 1_G)<\varepsilon$ for $0\leq k<n$. In other words, by definition of the set $W_{\varepsilon,n}$, we have 
that $g_j(x)\in W_{\varepsilon,n}$. It follows that
$$
\pi\bigl(\bigcup_{j=1}^N W_{\varepsilon,n}h_j\bigr)=X.
$$
Therefore, by choice of the normalization of $\lambda$,
$$
1\leq \lambda\bigl(\bigcup_{j=1}^N W_{\varepsilon,n}h_j\bigr)\leq
N\lambda(W_{\varepsilon,n})\leq N C'\varepsilon^dn^{-p}, 
$$
where the inequality follows from Corollary~\ref{cor:measuW}.  
We conclude that $N>C\varepsilon^{-d}n^p$ for some constant $C$. \qed

\subsubsection{Upper bound of the complexity}
\label{subsec:uperbound}
Recall that $D$ is a Borel fundamental domain of the projection $\pi\colon G\to X$. 
 Let $D_1$ be a compact subset of $G$ containing all  points at a distance at most $1$ from $D$.
 We make use of the definition of complexity using $\varepsilon$-$n$ separate sets 
 (Section~\ref{subsec:itrotop}) to show:
 \begin{lemma} 
\label{Lem:coverD}
Assume that $\varepsilon<1$.
There exists a subset $\{h_1,\dots,h_N\}$ of $D$ with 
\begin{equation}
\label{eq:boundN}
N\leq\lambda(D_1)\lambda(W_{\varepsilon/2,n})^{-1}
\end{equation}
 such that 
\begin{equation}
\label{eq:coverD}
D\subset \bigcup_{j=1}^N W_{\varepsilon,n} h_j.
\end{equation}
\end{lemma}
\begin{proof}
Let $N$ be the maximal cardinality of a subset $\{h_1,\dots,h_N\}$ of $D$ such that the sets $W_{\varepsilon/2,n}h_j$ are disjoint.

Since every element $g\in W_{\varepsilon/2,n}$ satisfies $d_G(g,1_G)<\varepsilon/2<1$, all the subsets $W_{\varepsilon,n} h_j$ are contained in $D_1$. Therefore,
$$
\lambda(D_1)\geq\sum_{j=1}^N \lambda(W_{\varepsilon/2,n}h_j)=N\lambda(W_{\varepsilon,n}).
$$
We claim that the set $\{h_1,\dots,h_N\}$ satisfies~\eqref{eq:coverD}. Assume instead that 
this does not hold and that $h\in D$ does not belong to this union.   It follows immediately 
from the definition that $W_{\varepsilon/2,n}$ is symmetric and that 
$W_{\varepsilon/2,n}\cdot W_{\varepsilon/2,n}\subset W_{\varepsilon,n}$. Therefore, for $1\leq j\leq n$, since 
$h\notin W_{\varepsilon,n}h_j$,  we have that $W_{\varepsilon/2,n}h_j\cap W_{\varepsilon/2,n}h=\emptyset$. 

Setting $h_{N+1}=h$, we have that the set $\{h_1,\dots,h_N,h_{N+1}\}$ satisfies the imposed condition, contradicting the maximality of $N$.
\end{proof}

We now show that the upper bound of Theorem~\ref{th:main} holds, thereby completing the proof.  
Let $\varepsilon>0$ be sufficiently small.
Let $N$ and $\{h_1,\dots,h_N\}$ be defined as in the conclusion of Lemma~\ref{Lem:coverD}.

Let $x_j=\pi(h_j)$ for $1\leq j\leq N$.
 We claim that  $\{x_1,\dots,x_N\}$ is an $\varepsilon$-$n$ spanning set for $X$.
 
 Let $x\in X$ and $h\in D$ be such that $\pi(h)=x$. There exists $j$ with $1\leq j\leq N$ such that 
 $h\in W_{\varepsilon,n}h_j$, meaning that there exists $g\in W_{\varepsilon,n}$ with 
 $h=gh_j$. For $0\leq k<n$, we have
\begin{align*}
d_X(T^kx,T^kx_j)
&=d_X(\pi(\tau^k gh_j),\pi(\tau^kh_j))
\leq d_G(\tau^kgh_j,\tau^kh_j)\\
& =d_G(\tau^kg,\tau^k)
=d_G(1_G,\tau^k g\tau^{-k})<\varepsilon,
\end{align*}
since $g\in W_{\varepsilon,n}$. This proves the claim. 

On the other hand,
by~\eqref{eq:boundN} and Corollary~\ref{cor:measuW} , we have that $N\leq C\varepsilon^{-d}n^p$ for some $C>0$, concluding the proof of Theorem~\ref{th:main}.
 
\section{Some topological applications}
\label{sec:top_app}

As in the ergodic setting, we find classes of systems such  that all factors $Z_{\rm top,s}(X)$, $s \geq 1$, of  
$(X,T)$ are equal to $Z_{\rm top, 1}(X)$.

It follows directly from Corollary~\ref{cor:spectrum} that transitive systems of finite topological rank satisfy this property; in particular, this is the case for minimal substitution dynamical systems and minimal interval exchange transformations.  Namely, let $(X,T)$ be such a system 
with nilfactor $(Y,S)$. Since $(Y,S)$ is transitive, it is  uniquely ergodic (see~\cite{AGH} and~\cite{P}), and its invariant measure is the Haar measure $\nu$ of $Y$. Let $\mu$ be an invariant ergodic measure on $(X,T)$. Then $(X,\mu,T)$ is a  system of measure theoretical finite rank, and the topological factor map $X\to Y$ is also a measure-theoretic factor map.  Thus by Corollary~\ref{cor:spectrum}, $(Y,\nu,S)$ is  measure theoretically isomorphic to a rotation.  By the rigidity properties of nilsystems (see for example~\cite[Appendix A]{HKM}),
$(Y,S)$ is topologically isomorphic to this rotation.

Corollary~\ref{cor:main} can also be used to find other such classes:

\begin{proposition}
\label{prop:symbol}
Let $(X,T)$ be a transitive subshift  and assume that 
$$\liminf_{n\to+\infty}\frac 1n \MC_X(n)<+\infty.$$
Then $(X,T)$ does not admit any topological nilfactor other than  rotations. 
Therefore, for every $s\geq 1$, the topological factor 
$Z_{\text{\rm top},s}(X)$ of $X$ is equal to its topological Kronecker factor $Z_{\text{\rm top},1}(X)$. 

More generally, if for some $s\geq 1$ we have
\begin{equation}
\label{eq:smaalcombinat}
\liminf_{n\to+\infty}\frac 1{n^{s}} \MC_X(n)<+\infty,
\end{equation}
then $(X,T)$ does not admit any  nilsystem as a topological factor that is not an $s$-step nilsystem. Therefore, for every $t\geq s$, the topological factor
$Z_{\text{\rm top},t}(X)$ of $X$ is equal to $Z_{\text{\rm top},s}(X)$.
\end{proposition}

\begin{proof}
 The last statement follows immediately from the fact that $Z_{\text{\rm top},s}(X)$ is the inverse limit of all $s$-step topological nilfactors of $(X,T)$. 
 
Assuming~\eqref{eq:smaalcombinat}, it suffices to show that
there is no 
topological factor map $\phi\colon (X,T)\to(Y,S)$, where $(Y,S)$ is a minimal $(s+1)$-step nilsystem that is not an $s$-step nilsystem.
Indeed, any minimal nilsystem that is not an $s$-step nilsystem admits an $(s+1)$-step nilsystem
as a factor that is not an $s$-step nilsystem. 

Assume instead that such a factor map $\phi\colon (X,T)\to(Y,S)$  exists and let $d_Y$ be a distance on $Y$ defining its topology.
Recall that $(X,T)$ is a transitive subshift on the finite alphabet $A$.  We write $x\in X$  as $x=(x_n\colon n\in\Z)$, and if $I\subset\Z$ is a finite interval, then we write
$x_I$ for the finite sequence $(x_n\colon n\in I)$.

Let $\varepsilon>0$. Since $\phi$ is continuous, there exists an integer $L=L(\varepsilon)>0$ such that 
$d_{Y}(\phi(x),\phi(y))<\varepsilon$ whenever $x,y\in X$ satisfy $x_{[-L,L]}=y_{[-L,L]}$.
Therefore, if $x,y\in X$ satisfy $x_{[-L,n+L]}=y_{[-L,n+L]}$ for some $n\geq 1$, then $d_{Y}(T^k\phi(x),T^k\phi(y))<\varepsilon$ for every $k\in\{0,\dots, n\}$.

Fix $n\geq 1$ and set $m=\MC_X(n+2L+1)$. 

By definition, there exist $m$ elements $x^{(1)}, \dots, x^{(m)}$ of $X$ such that for every $x\in X$, there exists $j\in\{1,\dots,m\}$ with $x_{[-L,n+L]}=x^{(j)}_{[-L,n+L]}$.
By definition of $L$ and since $\phi$ is onto, the set $\{\phi(x^{(j)})\colon 1\leq j\leq m\}$ is an $\varepsilon$-$n$ spanning set of $(Y,S)$. 
Thus
$$
\MS_{Y,S,d_{Y}}(\varepsilon,n)\leq \MC_X(n+2L+1) .
$$
Since this holds for every $n\geq 1$, it follows from the hypothesis that
$$
\liminf_{n\to+\infty}\frac 1{n^s} \MS_{Y,S,d_{Y}}(\varepsilon,n)\leq 
\liminf_{n\to+\infty} \frac 1{n^s} \MC_X(n)<+\infty.
$$
But this contradicts the statement of Corollary~\ref{cor:main}.
\end{proof}

\begin{remark}
Under the  hypothesis of ``linear complexity,'' that is, that there exists a constant $c>0$ such that $\MC_X(n)\leq cn$ for every $n\in\N$, the first statement of 
Proposition~\ref{prop:symbol}
 can also be deduced from Corollary~\ref{cor:spectrum}, by the method discussed 
at the beginning of this section;
in this case, the system $(X,T)$ has topological finite rank~\cite{F3}.
\end{remark}

\appendix

\section{Proof of Proposition~\ref{prop:lebesgue}}
\label{sec:prooflebesgue}

For convenience, we repeat the statement of Proposition~\ref{prop:lebesgue}:
\begin{proposition*}
Let $(X=G/\Gamma,\mu,T)$ be an 
ergodic $2$-step nilsystem that is not a rotation. Then $L^2(\mu)$ can be written as the orthogonal sum $L^2(\mu)=\CH\oplus\CH'$ of two closed $T$-invariant subspaces such that the restriction of $T$ to $\CH$ has discrete spectrum and its restriction to $\CH'$ has Lebesgue spectrum of infinite multiplicity.
\end{proposition*}

\subsection{Reductions and notation}

Before proving the proposition, we make some usual reductions that present the nilmanifold
in a standard way. Note that these reductions are not the same as those of Section~\ref{subsec:nilsystems}.

Let $(X=G/\Gamma,\mu,T)$ be an 
ergodic $2$-step nilsystem that is not a rotation.
As in Section~\ref{subsec:nilsystems}, we write $\tau\in G$ for the element defining the transformation $T$ and $G_0$ for  the connected component 
of the identity. By minimality we can assume that the subgroup spanned by $G_0$ and $\tau$ is dense in $G$.  This implies that $G_2$ is connected and thus included in $G_0$.

Let $\Gamma'$ be the largest normal subgroup of $G$ contained in  $\Gamma$. By substituting $G/\Gamma'$ for $G$ and $\Gamma/\Gamma'$ for $\Gamma$, we reduce to the case that $\Gamma$ does not contain any normal subgroup of $G$. 
Thus the action of $G$ on $X$ is faithful and it follows that $\Gamma$ 
is abelian and that $G_2$ is compact.  Since this Lie group is abelian and connected,   
it is a finite dimensional torus.
 Furthermore, it follows that each 
of subgroups spanned by $[\tau,\Gamma]$, by $[\tau,G]$ and  by $[\tau,G_0]$ is dense in $G_2$.

Let $q\colon G\to G/G_2$ and $\pi\colon G\to X$  denote the quotient maps. 
We recall   (see~\cite{P} and~\cite{L2}) that the Kronecker factor of $X$ is 
the compact abelian Lie Group $Z_1=G/G_2\Gamma$, endowed with its Lebesgue measure 
$m_{Z_1}$
 and translation by $\alpha=p\circ\pi(\tau)$, where 
$p\colon X\to Z_1$ denotes the factor map.

We remark that for every $a\in G$, the map $g\mapsto [a,g]$ is a group homomorphism from $G$ to $G_2$ (see Section~\ref{subsec:nilsystems}) and that the kernel of this homomorphism contains $G_2$.

\subsection{Lebesgue spectrum}
For $\chi\in\wh{G_2}$, set 
$$
\CH_\chi
=\bigl\{f\in L^2(\mu)\colon f(u\cdot x)=\chi(u)f(x)\text{ for every }u\in G_2\text{ and $\mu$-a.e. $x$}\bigr\}.
$$
Each space $\CH_\chi$ is invariant under  $T$ and $L^2(\mu)$ is the orthogonal sum of these spaces.  We have that  $\CH_1$ is the space of functions that factorize through $Z_1$,  and 
each  space 
$\CH_\chi$ is invariant under multiplication by functions belonging to $L^\infty(\mu)\cap\CH_1$. 

We show: 
\begin{lemma}
\label{lem:lebesgue1}
Let $\chi$ be a nontrivial character of $G_2$.  Then the spectral measure 
associated to any  function in $\CH_\chi$ is absolutely continuous with respect to  the
Lebesgue measure $m_\T$ of $\T$.
\end{lemma}
\begin{proof}
 Let $\chi$ be a nontrivial character of $G_2$.

The group homomorphism $g\mapsto \chi[g,\tau]$ factorizes through $G/G_2$ and so there exists a character $\widetilde\chi$ of $G/G_2$ satisfying
\begin{equation}
\label{eq:deftildechi}\chi[g,\tau]=\widetilde\chi\circ q(g)\text{ for every }g\in G.
\end{equation}

Let $\gamma\in\Gamma$. Since $\gamma$ commutes with $\Gamma$ and with $G_2$, the map $g\mapsto[\gamma,g]$ factorizes through the Kronecker factor $Z_1$.  Thus there exists a group homomorphism $\wh\gamma\colon Z_1\to G_2$ such that  
\begin{equation}
\label{eq:defwhgamma}
\wh\gamma\circ p\circ\pi(g)=[\gamma,g]\text{ for every }g\in G.
\end{equation} 

Let $K$ be a compact subset of $G$ such that the restriction to $K$ of the projection $\pi\colon G\to X$ is onto.
For every $n\in\Z$, we choose 
$\gamma_n$ and $g_n$ with  
\begin{equation}
\label{eq:defgn}
\gamma_n\in\Gamma,\ g_n\in K\text{ and }\tau^n=g_n\gamma_n.
\end{equation}

The family $L$ of characters $\widetilde\chi_{m,n}$ of $G/G_2$ defined by $\widetilde\chi_{m,n}\circ q(g)=\chi[g,g_ng_m^{-1}]$ is a bounded subset of   $\wh{G/G_2}$.
But the closed subgroup of $\wh{G/G_2}$ spanned by $\widetilde\chi$ is not compact; 
if not, the restriction of $\widetilde\chi$
to some open subgroup $H$ of $G/G_2$ would be trivial and so 
the restriction of $\widetilde\chi\circ q$ to $G_0$ would also be trivial, 
implying that $\chi([g,\tau])=1$ for every $g\in G_0$ and $\chi$ would be  trivial 
(again, the subgroup spanned by $[\tau, G_0]$ is dense in $G_2$), a contradiction.
Thus this group is discrete and in particular the set 
$$
\Lambda:=\{k\in\Z\colon \widetilde\chi^k\in L\}
$$
is finite. 

For every $n\in\Z$, write
\begin{equation}
\label{eq:defthetan}
\theta_n =\chi\circ\wh{\gamma_n}.
\end{equation}
We claim that for every $\theta\in\wh{Z_1}$,
\begin{equation}
\label{eq:cardLambda} 
\bigl|\{ n\in\Z\colon \theta_n=\theta\}\bigr|\leq|\Lambda|.
\end{equation}
To check this,  let $m$ and $n$ be integers with  $\theta_m=\theta_n$. By definitions~\eqref{eq:defthetan} and~\eqref{eq:defwhgamma}, for every $g\in G$ we have
$\chi([\gamma_m,g])=\chi([\gamma_n,g])$ and, by the choice~\eqref{eq:defgn} of $\gamma_n$ and $g_n$, we have
 $\chi(\tau^{m-n},g])=\chi([g_ng_m^{-1},g])=\widetilde\chi_{m,n}\circ q(g)$. By~\eqref{eq:deftildechi}, we conclude that $\widetilde\chi^{m-n}=\widetilde\chi_{m,n}\in\Lambda$ and the claim follows.

Let $f\in \CH_\chi$ be a function 
belonging to  the space $\CC^k(X)$ of $k$-times differentiable functions on $X$, for some $k$ to be defined later.
Note that  for every $n\in\Z$, the function  $x\mapsto f(g_n\cdot x)\cdot\overline{f(x)}$ belongs to the space 
$\CH_1$ and can be written as $h_n\circ p$ for some function $h_n$ on $Z_1$. 
 Since all the elements $g_n$ belong to the compact subset $K$ of $G$ and the action of $G$ by translation on $\CC^k(X)$ is continuous (with respect to the usual topology of $\CC^k(X)$),  all functions $x\mapsto f(g_n\cdot x)\cdot\overline{f(x)}$ belong to some compact subset of $\CC^k(X)$.  It follows that all the functions $h_n$ belong to some compact subset of $\CC^k(Z_1)$. Taking $k$ to be sufficiently large,  this implies that
\begin{equation}
\label{eq:sumsuptheta}
\sum_{\theta\in\wh{Z_1}}\sup_{n\in\Z}\bigl|\wh{h_n}(\theta)\bigr|<+\infty.
\end{equation}

The Fourier-Stieljes transform of the spectral measure $\sigma_f$ of $f$ is given by:
\begin{align*}
\wh{\sigma_f}(n)
&:=
\int f(T^nx)\cdot \overline{f(x)}\,d\mu(x)
=\int f(g_n\gamma_n\cdot x)\cdot \overline{f(x)}\,d\mu(x)
\text{ by~\eqref{eq:defgn}}\\
&=\int f(g_n\wh{\gamma_n}\circ p(x)\cdot x)\cdot \overline{f(x)}\,d\mu(x)
\text{ by~\eqref{eq:defwhgamma}}\\
&
=\int\chi\circ\wh{\gamma_n}\circ p(x)\cdot f(g_n\cdot x)\cdot \overline{f(x)}\,d\mu(x)
\text{ by the definition of $\CH_\chi$}\\
&=\int_{Z_1} \theta_n(z)
h_n(z)\,dm_{Z_1}(z)\text{ by~\eqref{eq:defthetan} and the definition of $h_n$.}
\end{align*}

We deduce:

\begin{align*}
\sum_{n\in\Z}\bigl|\wh{\sigma_f}(n)\bigr|
&=\sum_{n\in\Z}\bigl|\wh{h_n}(\theta_n^{-1})|\\
&\leq 
\sum_{\theta\in\wh{Z_1}}\bigl|\{n\in\Z\colon \theta_n=\theta\}\bigr|\,
\sup_{n\in\Z}\bigl|\wh{h_n}(\theta)\bigr|\\
&\leq |\Lambda|\sum_{\theta\in\wh{Z_1}}\sup_{n\in\Z}\bigl|\wh{h_n}(\theta)\bigr|<+\infty
\text{ by~\eqref{eq:cardLambda} and~\eqref{eq:sumsuptheta}.} 
\end{align*}

Therefore, the spectral measure of $f$ is absolutely continuous with respect to 
the Lebesgue measure $m_\T$ of $\T$.  By density, this property extends to every function in $\CH_\chi$.
\end{proof}

We use this to complete the proof of Proposition~\ref{prop:lebesgue}.
\subsection{End of the proof of Proposition~\ref{prop:lebesgue}}
For $\chi=1$, note that a function belonging to $\CH_1$ has discrete spectral measure.

By ergodicity, $\{\alpha^n\colon n\in\Z\}$ is dense in $Z_1$. The group of eigenvalues of $(X,T)$ is  $E:=\{\theta(\alpha)\colon\theta\in\wh Z\}$. Since $Z_1$ admits a (nontrivial) torus as an open subgroup, $E$ is dense in $\T$.

Let $\chi$ be a nontrivial character of $G_2$. For every $\theta\in\wh{Z_1}$,   the space $\CH_\chi$ is invariant under multiplication by the function $\theta\circ p$, as 
this function belongs to $L^\infty(\mu)\cap\CH_1$. For $f\in\CH_\chi$, the spectral measure of $f\cdot \theta\circ p$ is equal to the image of $\sigma_f$ under translation by $\theta(\alpha)$. Therefore, the maximal spectral type $\sigma_\chi$ of the restriction of $T$ to $\CH_\chi$ is quasi-invariant under translation by 
$\lambda$ for every $\lambda\in E$, meaning that if $A\subset \T$ satisfies $\sigma_\chi(A)=0$, then $\sigma_\chi(A+\lambda)=0$.  But $\sigma_\chi$ is absolutely continuous with respect to Lebesgue measure and so $\sigma_\chi$ is equivalent to Lebesgue measure.  It follows that there exists a function $f_\chi\in\CH_\chi$ such that $\sigma_{f_\chi}=m_\T$.

On the other hand, the invariant spaces $\CH_\chi$, for $\chi\neq 1\in\wh{G_2}$,  are mutually orthogonal, completing the proof.
\qed

\section{Proof of Proposition~\ref{prop:lineaire}}
\label{ap:linear}

For convenience, we repeat the statement of Proposition~\ref{prop:lineaire}:
\begin{proposition*}
Let $\R^d$ be endowed with the Euclidean norm $\norm\cdot$ and let
the Lebesgue measure of a Borel subset $K$ of $\R^d$ be written $|K|$. 
Let $A$ be a $d\times d$ matrix and assume that it is unipotent, meaning that $(\id-A)^d=0$. For every integer $n\geq 2$, let 
$$
\CW_n=\bigl\{ \xi\in\R^d\colon \norm{A^k\xi}\leq 1\text{ for }1\leq k<n\bigr\}.
$$
If
$$
p=\sum_{k=1}^{d-1}\dim(\range(\id-A)^k), 
$$
there exist positive constants $C$ and $C'$ (depending on $d$ and on $A$) such that
$$
C n^{-p} \leq |\CW_n|\leq C'n^{-p}
$$
for every $n$.
\end{proposition*}

\begin{notation}
Let $J_r$ denote the $r\times r$ upper triangular elementary Jordan matrix whose entries are given by 
$$
J_{r,i,j}=\begin{cases}
1 & \text{for } 1\leq i\leq r\text{ and }j=i ;\\
1& \text{for } 1\leq i\leq r-1\text{ and }j=i+1;\\
0 &\text{otherwise.}
\end{cases}
$$
In other words, the matrix $J_r$ has $1$'s on the diagonal and on the superdiagonal, and $0$'s elsewhere.
\end{notation}

We begin with a lemma: 
\begin{lemma}
There exists a constant $C=C(r)$ such that 
\begin{multline}
\label{eq:xsmall}
\text{ if }|x_j|\leq \frac 1{n^{j-1}}\text{ for some $n\geq 2$ and }1\leq j\leq r,\\
\text{ then }
\Bigl|\sum_{j=1}^r(J_r^k)_{i,j}x_j\Bigr|\leq C\text{ for }1\leq i\leq r\
\text{ and }0\leq k< n.
\end{multline}
On the other hand, there exists a constant $C'=C'(r)>0$ such that
\begin{multline}
\label{eq:Jxsmall}
\text{if }\Bigl|\sum_{j=1}^r(J_r^k)_{i,j}x_j\Bigr|\leq 1
\text{ for }1\leq i\leq r\text{ and }0\leq k< n\text{ for some }n\geq 2,\\
\text{ then }|x_j|\leq \frac{C'}{n^{j-1}}\text{ for }1\leq j\leq r.
\end{multline}
\end{lemma}

\begin{proof}
For $k\geq 1$,
\begin{equation}
\label{eq:Jkij}
(J_r^k)_{i,j}=\binom k{j-i},
\end{equation}
where we make use of the convention that $\binom kp=0$ if $p<0$ or $p>k$.

To prove the first statement, assume  that $x_1,\dots,x_r$  satisfy the hypothesis of~\eqref{eq:xsmall}.
Then for $1\leq i\leq r$, 
$$
\Bigl|\sum_{j=1}^r(J_r^k)_{i,j}x_j\Bigr|\leq 
\sum_{j=i}^r \binom k{j-i}k^{-j+1}\leq C_1(r)\sum_{j=i}^rk^{j-i}k^{-j+1}\leq C_2(r), 
$$
completing the proof of~\eqref{eq:xsmall}.

The proof of~\eqref{eq:Jxsmall} requires more work. 
Assume that $x_1,\dots,x_r$ satisfy the hypothesis of~\eqref{eq:Jxsmall}. 

Taking $k=0$, we have that 
\begin{equation}
\label{eq:xj}
|x_j|\leq 1\text{ for }1\leq j\leq r.
\end{equation}

Thus, without loss of generality, we can restrict ourselves to the case that $n$ is sufficiently large, 
and assume that $n\geq 2^r+1$. Define the integer $q\geq 1$ by
\begin{equation}
\label{eq:nrq}
2^rq\leq n-1 <2^r(q+1).
\end{equation}
In the sequel, we only make use of hypothesis~\eqref{eq:Jxsmall} with $i=1$ and $k=2^mq$ with $0\leq m \leq r$. Formula~\eqref{eq:Jkij} for the coefficients of the matrix $J_r^k$ gives
$$
\Bigl|\sum_{j=1}^r\binom{2^mq}{j-1}x_j\Bigr|\leq 1 \text{ for } 0\leq m\leq r.
$$
 Since $\binom{2^mq}0=1$ and $|x_1|\leq 1$,  by~\eqref{eq:xj} we have that
 \begin{equation}
\label{eq:binom}
\Bigl|\sum_{j=2}^r\binom{2^mq}{j-1}x_j\Bigr|\leq 2 \text{ for } 0\leq m\leq r.
\end{equation}
Define
\begin{equation}
\label{eq:p}
p_{2,j}(k)=\frac 1k \binom k{j-1}\text{ for }2\leq j\leq r.  
\end{equation}
Then $p_{2,j}$ is a polynomial of degree exactly $j-2$ in the variable $k$. 
Formula~\eqref{eq:binom} implies that
\begin{equation}
\label{eq:pr}
\bigl|\sum_{j=2}^r p_{2,j}(2^mq)x_j\Bigr|\leq \frac 2{2^mq}\text{ for }0\leq m\leq r.
\end{equation}

We continue by induction and assume that for some $\ell$ with $2\leq \ell<r$, we have
\begin{equation}
\label{eq:prell}
\Bigl|\sum_{j=\ell}^r p_{\ell,j}(2^mq)x_j\Bigr|\leq \frac C{(2^mq)^{\ell-1}}\text{ for }\ell-2\leq m\leq r,
\end{equation}
where $p_{\ell,j}(k)$ is a polynomial of degree  exactly $j-\ell$ for $\ell\leq j\leq r$.

The same formula applied with $m-1$ substituted for $m$ leads to 
\begin{equation}
\label{eq:prell2}
\Bigl|\sum_{j=\ell}^r p_{\ell,j}(2^{m-1}q)x_j\Bigr|\leq \frac {2^{\ell-1}C}{(2^mq)^{\ell-1}}\text{ for }\ell-1\leq m\leq r.
\end{equation}
For $\ell+1\leq j\leq r$, define
$$
p_{\ell+1,j}(k)=\frac 1k \bigl(p_{\ell,j}(k)-p_{\ell,j}(k/2)\bigr).
$$
Then $p_{\ell+1,j}(k)$ is a polynomial of degree $\leq j-\ell-1$ in the variable $k$. 
 In fact, this polynomial has degree 
exactly  $j-\ell-1$, as the coefficients of maximal degree
of $p_{\ell,j}(k)$ and $p_{\ell,j}(k/2)$ are not the same.

Taking the difference between the formulas~\eqref{eq:prell} and~\eqref{eq:prell2}, the constant term $p_{\ell,\ell}(2^mq)-p_{\ell,\ell}(2^{m-1}q)$ vanishes and we have that 
\begin{equation}
\label{eq:prellplus1}
\bigl|\sum_{j=\ell+1}^r p_{\ell+1,j}(2^mq)x_j\Bigr|\leq \frac C{(2^mq)^{\ell}}\text{ for }\ell-1\leq m\leq r.
\end{equation}

By induction, Inequality~\eqref{eq:prell2} is proven for $2\leq\ell<r$.  
The polynomial $p_{\ell,\ell}$ is a  nonzero constant 
and we have that 
\begin{equation}
\label{eq:xr}
|x_r|\leq\frac C{(2^rq)^{r-1}}.
\end{equation}

By backwards induction, we now show that 
\begin{equation}
\label{eq:xell}
|x_\ell|\leq \frac C{(2^rq)^{\ell-1}} \text{ for }1\leq \ell\leq r.
\end{equation}
For $\ell=r$, this is exactly~\eqref{eq:xr}. Assume that $1\leq\ell < r$ and that this
bound holds for $\ell+1,\ell+2,\dots,r$. By~\eqref{eq:prell} applied with $m=r$, 
$$
|x_\ell|\leq \frac C{(2^rq)^{\ell-1}}+
\sum_{j=\ell+1}^r |p_{\ell,j}(2^rq)|\frac C{(2^rq)^{j-1}}.
$$
Since $p_{\ell,j}$ is a polynomial of degree $j-\ell$, we have $|p_{\ell,j}(k)|\leq Ck^{j-\ell}$ for some $C>0$ and~\eqref{eq:xell} follows.

We conclude the proof by using~\eqref{eq:nrq} to conclude that $2^rq> (n-1)q/(q+1)\geq n/4$.
\end{proof}

\begin{corollary}
\label{cor:measure}
Let $\R^r$ be endowed with the supremum norm $\norm\cdot_\infty$ and $|K|$ denote the Lebesgue measure of a Borel subset $K$ of $\R^r$. Then for every $n\geq 2$, we have
$$
Cn^{-r(r-1)/2}\leq 
\Bigl|\bigl\{ x\in\R^r\colon \norm{J_r^kx}_\infty\leq 1\text{ for }0\leq k<n\bigr\}\Bigr|
\leq C' n^{-r(r-1)/2}.
$$
\end{corollary}

Using this, we complete the proof of Proposition~\ref{prop:lineaire}.
Since the matrix $A$ is unipotent,
there  exists a $d\times d$ invertible matrix $\Phi$ such that the matrix
$B:=\Phi A\Phi^{-1}$ is in Jordan form. Thus all the entries of $B$
are all equal to $0$ other than $m\geq 1$ diagonal square blocks, each of which is an elementary Jordan $r_j\times r_j$ matrix $J_{r_j}$ for $1\leq j\leq m$.

We have
$$
|\CW_n|=|\det(\Phi)|^{-1}\,|\Phi(\CW_n)|
$$
for every $n$. Moreover,
$$
\Phi(W_n)=\bigl\{ \eta\in\R^d\colon \norm{\Phi B^k\eta}\leq 1\text{ for }0\leq k<n\bigr\}
$$
and thus there  exist positive constants $c$ and $c'$ with 
\begin{multline*}
c\bigl\{ \eta\in\R^d\colon \norm{ B^k\eta}_\infty\leq 1\text{ for }0\leq k<n\bigr\}
\subseteq \Phi(\CW_n)
\\
\subseteq
c'\bigl\{ \eta\in\R^d\colon \norm{ B^k\eta}_\infty\leq 1\text{ for }0\leq k<n\bigr\}.
\end{multline*}
Therefore
\begin{multline*}
c^d\Bigl|\bigl\{ \eta\in\R^d\colon \norm{ B^k\eta}_\infty\leq 1\text{ for }0\leq k<n\bigr\}
\Bigr|
\leq \bigl|\Phi(\CW_n)\bigr|
\\
\leq
c'^d\Bigl|
\bigl\{ \eta\in\R^d\colon \norm{ B^k\eta}_\infty\leq 1\text{ for }0\leq k<n\bigr\}
\Bigr|.
\end{multline*}
On the other hand, 
\begin{multline*}
\bigl\{ \eta\in\R^d\colon \norm{ B^k\eta}_\infty\leq 1\text{ for }0\leq k<n\bigr\}\\
=\prod_{j=1}^m\bigl\{\eta\in\R^{r_j}\colon \norm{J_{r_j}^k\eta}_\infty\leq 1\text{ for } 0\leq k<n\bigr\}.
\end{multline*}
Combining these remarks with Corollary~\ref{cor:measure}, we have that
$$
Cn^{-p}\leq |\CW_n|\leq C' n^{-p}, 
$$
where 

\begin{multline*}
p=
\sum_{j=1}^m \frac{r_j(r_j-1)}2 =
\sum_{j=1}^m\sum_{k=1}^{r_j-1}k\\
\quad\begin{aligned}
\quad
&=\sum_{j=1}^m \sum_{k=1}^{d-1} k\bigl(\dim(\ker(\id-J_{r_j})^{k+1})
-\dim(\ker(\id-J_{r_j})^{k})\bigr)\\
 &=\sum_{k=1}^{d-1}k\bigl(\dim(\ker(\id-A)^{k+1})-\dim(\ker(\id-A)^k)\bigr)\\
&=\sum_{k=1}^{d-1}\dim(\range(\id-A^k)).\hskip5cm\qed
\end{aligned}
\end{multline*}


\begin{thebibliography}{99}

\bibitem{AGH}
{\sc L. Auslander, L. Green \& F. Hahn. }
{\it Flows on homogeneous spaces.}
 Ann. Math. Studies
{\bf 53}, Princeton Univ. Press, 1963.

\bibitem{B}
{\sc V.~Bergelson.}
Weakly mixing PET.
{\it  Erg. Th. \& Dyn. Sys.}, {\bf 7}  (1987), 337--349.


\bibitem{BHK}
{\sc V. Bergelson, B. Host \& B. Kra,} with an appendix by I.~Ruzsa.
Multiple recurrence and nilsequences.
 {\it Invent. Math.}, {\bf 160} (2005), 261--303.
 
\bibitem{BLL}
{\sc V. Bergelson, A. Leibman, A \& E. Lesigne,} 
Intersective polynomials and the polynomial Szemer\'edi theorem. 
{\it Adv. Math.},  {\bf 219} (2008), no. 1, 369--388.

 
\bibitem{BHM}
{\sc F. Blanchard, B. Host \& A. Maass.}  Topological complexity. {\it Erg. Th. \& Dyn. Sys.}, {\bf 20} (2000), no. 3, 641–-662.

 \bibitem{BDM}
{\sc  X. Bressaud,  F. Durand \&  A. Maass.}
On the eigenvalues of finite rank Bratteli-Vershik dynamical systems. 
{\it  Erg. Th. \& Dyn. Sys.}, {\bf 30} (2010), no. 3, 639-–664. 



 



\bibitem{C}
{\sc Q. Chu.}
Convergence of weighted polynomial multiple ergodic averages. 
{\it Proc. Amer. Math. Soc.}, {\bf 137} (2009), no. 4, 1363–-1369. 

\bibitem{D}
{\sc F.  Durand.}
Linearly recurrent subshifts have a finite number of non-periodic subshift 
factors. 
{\it Erg. Th. \& Dyn. Sys.}, {\bf 20} (2000), no. 4, 1061–-1078. 


\bibitem{DDMSY}
{\sc P. Dong, S. Donoso, A. Maass, S. Shao \& X. Ye.} 
Infinite-step nilsystems, independence and complexity. 
To appear in {\it Erg. Th. \& Dyn. Sys.}, 
\href{http://arxiv.org/abs/1105.3584}{\tt arXiv:1105.3584}

\bibitem{F1} 
{\sc S. Ferenczi.}
Systèmes localement de rang un.  {\it Ann. Inst. H. Poincaré Probab. Statist.} 
{\bf 20} (1984), no. 1, 35–-51.

\bibitem{F2}
{\sc Ferenczi, S.}
Systèmes de rang un gauche. 
{\it Ann. Inst. H. Poincaré Probab. Statist.} {\bf 21} (1985), no. 2, 177–-186. 


\bibitem{F3} 
{\sc S. Ferenczi.}
Rank and symbolic complexity.  
{\it  Erg. Th. \& Dyn. Sys.}, {\bf 16} (1996), no. 4, 663–-682. 


\bibitem{F4} 
{\sc S. Ferenczi.}
Measure-theoretic complexity of ergodic systems. 
{\it Israel J. of Math.}, {\bf 100} (1997), no. 4, 189--207. 


\bibitem{Fr1}
{\sc N. Frantzikinakis}. 
Multiple ergodic averages for three polynomials and applications. 
{\it Trans. Amer. Math. Soc.},  {\bf 360} (2008), no. 10, 5435--5475. 

\bibitem{Fr2}
{\sc N. Frantzikinakis}. 
Multiple recurrence and convergence for Hardy field sequences of polynomial growth. 
{\it J. d'Analyse Math.}, {\bf 112}  (2010), 79--135


\bibitem{FrK}
{\sc N. Frantzikinakis \& B. Kra}.
Ergodic averages for independent polynomials and applications.  
{\it J. Lond. Math. Soc.}, {\bf 74} (2006), 131--142.

\bibitem{FrW}
{\sc N. Frantzikinakis \& M. Wierdl}. 
A Hardy field extension of Szemeredi's theorem. 
{\em Adv. Math.}, {\bf 222} (2009), 1--43

\bibitem{GJ}
{\sc R. Gjerde \&  \O. Johansen.} 
Bratteli-Vershik models for Cantor minimal systems associated to 
interval exchange transformations. 
{\it Math. Scand.}, {\bf 90} (2002), no. 1, 87–-100. 

\bibitem{GT}
{\sc B. Green \& T. Tao.}
Linear equations in primes. 
{\it Annals of Math.}, {\bf 171} (2010), no. 3, 1753--1850.





\bibitem{HK1}
{\sc B. Host \&  B. Kra.}
 Nonconventional ergodic averages and nilmanifolds.
 {\it  Annals of Maths.}, {\bf 161} (2005), 397--488. 
 
\bibitem{HK2} 
{\sc B. Host \& B. Kra.}
Convergence of polynomial ergodic averages.
{\it Israel J. of Math.}, {\bf 149} (2005), 1--20.

\bibitem{HK3} 
 {\sc B. Host \& B. Kra.}
Uniformity seminorms on $\ell^\infty$ and applications.
 {\em J. Anal. Math.}, {\bf 108}  (2009), 219--276. 

 
 \bibitem{HKM}
{\sc B. Host, B. Kra  \& A. Maass.}
Nilsequences and a topological structure theorem. 
 {\em Adv. Math.},  {\bf 224}  (2010),  no. 1, 103--129.
 
  \bibitem{KT}
{\sc A. Katok \&  J.-P. Thouvenot.}
Slow entropy type invariants and smooth realization of commuting measure-preserving transformations.
{\it Ann. Inst. H. Poincaré Probab. Statist.}, {\bf 33} (1997), no. 3, 323–-338.

\bibitem{L1}
{\sc A. Leibman.}
Polynomial sequences in groups.
{\it  J. Algebra} {\bf 201} (1998), no. 1, 189-–206. 
 
 \bibitem{L2}
{\sc A. Leibman.}
Pointwise convergence of ergodic averages for polynomial
sequences of translations on a nilmanifold. {\em  Erg. Th. \& Dyn. Sys.},
{\bf 25} (2005), no. 1, 201-213.


\bibitem{L3}
{\sc A. Leibman.}
Convergence of multiple ergodic averages along polynomials of several variables.
{\it Israel J. Math.}, {\bf 146} (2005), 303--316. 

\bibitem{M}
{\sc A. Malcev.} {\it On a class of homogeneous spaces.} Amer. Math. Soc. Transl. {\bf 39} (1951).



\bibitem{ORW} 
{\sc D. Ornstein,   D. Rudolph \&  B. Weiss.}
{\it Equivalence of measure preserving transformations.} 
Mem. Amer. Math. Soc., {\bf 37} (1982), no. 262.  


\bibitem{P}
{\sc W. Parry.}
Dynamical systems on nilmanifolds.
{\em Bull. London Math. Soc.}, 
{\bf 2} (1970), 37--40.

\bibitem{Q}
{\sc M. Queff\'elec.}
{\it Substitution dynamical systems — spectral analysis.}
Second edition. Lecture Notes in Mathematics, Vol. 1294. 
Springer-Verlag, Berlin, 2010.


\bibitem{SY}
{\sc S. Shao \& X. Ye.}  Regionally proximal relation of order d is an equivalence one for minimal systems
and a combinatorial consequence. 
\href{http://arxiv.org/abs/1007.0189} {\tt arXiv:1007.0189}.


\bibitem{Sta}
{\sc Starkov.} {\it Dynamical systems on homogeneous spaces.}
Translated from the 1999 Russian original by the author. Translations of Mathematical Monographs, 190. American Mathematical Society, Providence, RI, 2000.



\bibitem{St}
{\sc A.M. Stepin.}  Flows on solvable manifolds.  {\it Upekhi Mat. Nauk} {\bf 24} (1969), no. 2, 241-- 242 (Russian).

\bibitem{Wa}
{\sc P. Walters}  
{\it An Introduction to Ergodic Theory. }
Graduate Texts in Mathematics, {\bf 79}. Springer-Verlag, New York-Berlin, 1982.

\bibitem{We}
{\sc H. Weyl.} \'Uber die Gleichverteilung von Zahlen mod Eins. {\em Math. Ann.}, {\bf 77}
(1916), 313--352.

\end{thebibliography}
\end{document}